\documentclass[11pt,a4paper]{scrartcl}
\usepackage[utf8]{inputenc}
\usepackage[T1]{fontenc}
\usepackage{graphicx}
\usepackage{grffile}
\usepackage{longtable}
\usepackage{wrapfig}
\usepackage{rotating}
\usepackage[normalem]{ulem}
\usepackage{amsmath}
\usepackage{textcomp}
\usepackage{amssymb}
\usepackage{capt-of}
\usepackage{hyperref}

\usepackage[a4paper,verbose]{geometry}
\usepackage{libertine}

\usepackage[utf8]{inputenc}    
\usepackage[T1]{fontenc}

\usepackage{amsfonts}
\usepackage{amssymb}
\usepackage{amsthm}
\usepackage{amsmath}
\usepackage{caption}
\usepackage[inline]{enumitem}
\setlist{itemsep=0em, topsep=0em, parsep=0em}
\setlist[enumerate]{label=(\alph*)}
\usepackage{etoolbox}
\usepackage{stmaryrd}
\usepackage[dvipsnames]{xcolor}
\usepackage{ifthen}
\usepackage{suffix}
\usepackage{verbatim} %

\usepackage{xfrac}

\usepackage{bm}
\usepackage{physics}
\usepackage{mathtools} %
\usepackage{color}
\usepackage{wrapfig}
\usepackage{tikz}
\usepackage{tikzit}

\usepackage[framemethod=tikz]{mdframed}

\newcommand*{\secref}[1]{\S\ref{#1}}

\usepackage[backend=biber,style=numeric-comp,sorting=none,date=short,maxcitenames=2]{biblatex}

\usepackage{hyperref}
\hypersetup{colorlinks,linkcolor=darkblue,citecolor=darkblue,urlcolor=darkblue,breaklinks=true,pdfusetitle}

\newcommand*{\email}[1]{%
  \normalsize\href{mailto:#1}{\texttt{#1}}\par
}

\bibliography{bibliography}

\newcommand*{\citep}[1]{\parencite{#1}}
\newcommand*{\citet}[1]{\textcite{#1}}

\usepackage[capitalize]{cleveref} %

\usepackage{graphicx}
\usepackage{pict2e}

\def\mdash{{\hbox{-}}}

\makeatletter
\newcommand{\adjunction}{\@ifstar\named@adjunction\normal@adjunction}
\newcommand{\normal@adjunction}[4]{%
  #1\colon #2%
  \mathrel{\vcenter{%
    \offinterlineskip\m@th
    \ialign{%
      \hfil$##$\hfil\cr
      \longrightharpoonup\cr
      \noalign{\kern-.3ex}
      \smallbot\cr
      \longleftharpoondown\cr
    }%
  }}%
  #3 \noloc #4%
}
\newcommand{\named@adjunction}[4]{%
  #2%
  \mathrel{\vcenter{%
    \offinterlineskip\m@th
    \ialign{%
      \hfil$##$\hfil\cr
      \scriptstyle#1\cr
      \noalign{\kern.1ex}
      \longrightharpoonup\cr
      \noalign{\kern-.3ex}
      \smallbot\cr
      \longleftharpoondown\cr
      \scriptstyle#4\cr
    }%
  }}%
  #3%
}
\newcommand{\longrightharpoonup}{\relbar\joinrel\rightharpoonup}
\newcommand{\longleftharpoondown}{\leftharpoondown\joinrel\relbar}
\newcommand\noloc{%
  \nobreak
  \mspace{6mu plus 1mu}
  {:}
  \nonscript\mkern-\thinmuskip
  \mathpunct{}
  \mspace{2mu}
}
\newcommand{\smallbot}{%
  \begingroup\setlength\unitlength{.15em}%
  \begin{picture}(1,1)
  \roundcap
  \polyline(0,0)(1,0)
  \polyline(0.5,0)(0.5,1)
  \end{picture}%
  \endgroup
}
\makeatother

\newcommand\optar[2]{\ensuremath{\braket{#1\:}{\: #2}}}
\newcommand\bdag[1]{\ensuremath{#1_{(\cdot)}^\dag}}

\newcommand*\circled[3]
  {\tikz[baseline={([yshift=#2]current bounding box.center)}]{
      \node[shape=circle,draw,inner sep=#1] (char) {#3};}}

\newcommand{\docircL}[3]{\mathbin{\circled{#1}{#2}{\scalebox{#3}{\tiny{$\mathsf{L}$}}}}}
\newcommand{\docircR}[3]{\mathbin{\circled{#1}{#2}{\scalebox{#3}{\tiny{$\mathsf{R}$}}}}}

\newcommand*{\circL}{
  \mathchoice{\docircL{0.75pt}{-0.60ex}{1}}
             {\docircL{0.75pt}{-0.60ex}{1}}
             {\docircL{0.3pt}{-0.45ex}{0.70}}
             {\docircL{0.1pt}{-0.35ex}{0.50}}}
\newcommand*{\circR}{
  \mathchoice{\docircR{0.75pt}{-0.60ex}{1}}
             {\docircR{0.75pt}{-0.60ex}{1}}
             {\docircR{0.3pt}{-0.45ex}{0.70}}
             {\docircR{0.1pt}{-0.35ex}{0.50}}}

\let\op=\relax

\DeclareMathOperator{\op}{^\text{op}}

\newcommand{\nn}{{\mathbb{N}}}

\newcommand{\Cat}[1]{\mathbf{#1}}
\newcommand{\cat}[1]{\mathcal{#1}}
\newcommand{\Fun}[1]{\mathsf{#1}}

\newcommand{\Kl}{\mathcal{K}\mspace{-2mu}\ell}
\newcommand{\EM}{\mathcal{E}\mspace{-3mu}\mathcal{M}}

\let\Pr\relax
\DeclareMathOperator{\Pr}{P}
\DeclareMathOperator*{\E}{\mathbb{E}}

\renewcommand{\d}{\mathrm{d}}

\DeclareMathOperator{\dom}{dom}
\DeclareMathOperator{\cod}{cod}

\newcommand{\Pow}{\mathcal{P}}

\newcommand{\Dst}{\mathcal{D}}

\newcommand{\Giry}{\mathcal{G}}

\DeclareMathOperator{\id}{\mathsf{id}}

\DeclareMathOperator{\Set}{\Cat{Set}}

\newcommand{\xto}[1]{\xrightarrow{#1}}

\makeatletter
\newcommand{\mathoverlap}[2]{\mathpalette\mathoverlap@{{#1}{#2}}}
\newcommand{\mathoverlap@}[2]{\mathoverlap@@{#1}#2}
\newcommand{\mathoverlap@@}[3]{\ooalign{$\m@th#1#2$\crcr\hidewidth$\m@th#1#3$\hidewidth}}
\makeatother

\newcommand{\klcirc}{\bullet} %
\newcommand*{\smallklcirc}{\raisebox{0.18ex}{\scalebox{0.66}{$\klcirc$}}}
\newcommand{\klto}{\mathoverlap{\rightarrow}{\smallklcirc\,}}

\newcommand{\xklto}[1]{\mathoverlap{\xrightarrow{#1}}{\smallklcirc\,}}

\newcommand{\lenscirc}{
  \mathbin{\mathoverlap{\circ}{\raisebox{0.375ex}{\scalebox{1.0}[0.33]{$|$}}}}
}
\newcommand{\lensto}{\mathrel{\ooalign{\hfil$\mapstochar\mkern5mu$\hfil\cr$\to$\cr}}}

\makeatletter
\providecommand*{\xmapstofill@}{%
  \arrowfill@{\mapstochar\relbar}\relbar\rightarrow
}
\providecommand*{\xmapsto}[2][]{%
  \ext@arrow 0395\xmapstofill@{#1}{#2}%
}

\makeatletter
\def\slashedarrowfill@#1#2#3#4#5{%
  $\m@th\thickmuskip0mu\medmuskip\thickmuskip\thinmuskip\thickmuskip
   \relax#5#1\mkern-7mu%
   \cleaders\hbox{$#5\mkern-2mu#2\mkern-2mu$}\hfill
   \mathclap{#3}\mathclap{#2}%
   \cleaders\hbox{$#5\mkern-2mu#2\mkern-2mu$}\hfill
   \mkern-7mu#4$%
}
\def\rightslashedarrowfill@{%
  \slashedarrowfill@\relbar\relbar\mapstochar\rightarrow}
\newcommand\xslashedrightarrow[2][]{%
  \ext@arrow 0055{\rightslashedarrowfill@}{#1}{#2}}
\makeatother

\theoremstyle{definition}
\newtheorem{defn}{Definition}[section]

\newtheorem{ex}[defn]{Example}

\newtheorem{rmk}[defn]{Remark}

\newtheorem{prop}[defn]{Proposition}
\newtheorem{lemma}[defn]{Lemma}
\newtheorem{thm}[defn]{Theorem}
\newtheorem{cor}[defn]{Corollary}

\newtheorem*{thm*}{Theorem}
\newtheorem*{cor*}{Corollary}

\definecolor{darkblue}{rgb}{0,0,0.7}

\usepackage{authblk}
\author{Toby St. Clere Smithe}
\affil{Department of Experimental Psychology, \\ University of Oxford \\ \email{arxiv@tsmithe.net}}

\tikzstyle{xshiftu}=[shift = {(#1, 0)}]
\tikzstyle{yshiftu}=[shift = {(0, #1)}]

\tikzstyle{dot}=[inner sep=0.25mm,minimum width=1mm,minimum height=1mm,draw,shape=circle,text depth=-0.2mm]
\tikzstyle{white dot}=[dot,fill=white, draw=black]
\tikzstyle{action}=[dot,fill=white,scale=0.667,inner sep=0.5mm]
\tikzstyle{copier}=[dot,fill=white,scale=2.0]
\tikzstyle{black copier}=[dot,fill=black,scale=2.0]

\tikzstyle{box}=[fill=white, draw=black, shape=rectangle]
\tikzstyle{medium box}=[fill=white, draw=black, shape=rectangle, minimum width=1.5cm, minimum height=0.66cm]
\tikzstyle{arrow box}=[fill=white, draw, shape=rectangle,minimum height=5mm,yshift=-0.5mm,minimum width=5mm]

\tikzstyle{effect}=[regular polygon, regular polygon sides=3,draw]
\tikzstyle{state0}=[regular polygon, regular polygon sides=3,draw,shape border rotate=0]
\tikzstyle{state90}=[regular polygon, regular polygon sides=3,draw,shape border rotate=90]
\tikzstyle{state180}=[regular polygon, regular polygon sides=3,draw,shape border rotate=180]
\tikzstyle{state270}=[regular polygon, regular polygon sides=3,draw,shape border rotate=270]

\tikzstyle{scalar}=[diamond,draw,inner sep=1pt]

\tikzstyle{discarder}=[my ground,draw,inner sep=0pt,minimum width=4.2pt,minimum height=11.2pt,anchor=input,rotate=90]
\tikzstyle{discarder0}=[my ground,draw,inner sep=0pt,minimum width=4.2pt,minimum height=11.2pt,anchor=input,rotate=0]

\tikzstyle{pointy1}=[->]
\tikzstyle{midpoint1}=[-, {postaction={decorate,decoration={markings, mark=at position .5 with {\arrow{>}}}}}]
\tikzstyle{midpointy1pointy}=[->, {postaction={decorate,decoration={markings, mark=at position .5 with {\arrow{>}}}}}]
\tikzstyle{dashed1}=[-, dashed]
\tikzstyle{dotted1}=[-, dotted]
\tikzstyle{dash-pointy}=[->, dashed]

\input{strings.tikzdefs}

\ifexternalizetikz\tikzexternaldisable\fi
\newsavebox\sbground
\savebox\sbground{%
  \begin{tikzpicture}[baseline=0pt]
    \draw (0,-.1ex) to (0,.85ex)
    node[ground IEC,draw,anchor=input,inner sep=0pt,
    minimum width=3.15pt,minimum height=8.4pt,rotate=90] {};
  \end{tikzpicture}%
}
\newcommand{\ground}{\mathord{\usebox\sbground}}

\newsavebox\sbcopier
\savebox\sbcopier{%
  \begin{tikzpicture}[baseline=0pt]
    \node[copier,scale=0.7] (a) at (0,3.8pt) {};
    \draw (a) -- +(-90:.21);
    \draw (a) -- +(45:.21);
    \draw (a) -- +(135:.21);
  \end{tikzpicture}}
\newcommand{\copier}{\mathord{\usebox\sbcopier}}
\ifexternalizetikz\tikzexternalenable\fi

\newsavebox\bsbcopier
\savebox\bsbcopier{%
  \begin{tikzpicture}[baseline=0pt]
    \node[black copier,scale=0.7] (a) at (0,3.8pt) {};
    \draw (a) -- +(-90:.21);
    \draw (a) -- +(45:.21);
    \draw (a) -- +(135:.21);
  \end{tikzpicture}}
\newcommand{\bcopier}{\mathord{\usebox\bsbcopier}}
\ifexternalizetikz\tikzexternalenable\fi

\date{\today}
\title{Bayesian Updates Compose Optically}
\begin{document}

\maketitle
\begin{abstract}
Bayes' rule tells us how to invert a causal process in order to update our beliefs in light of new evidence. If the process is believed to have a complex compositional structure, we may ask whether composing the inversions of the component processes gives the same belief update as the inversion of the whole. We answer this question affirmatively, showing that the relevant compositional structure is precisely that of the \emph{lens} pattern, and that we can think of Bayesian inversion as a particular instance of a state-dependent morphism in a corresponding fibred category. We define a general notion of (mixed) Bayesian lens, and discuss the (un)lawfulness of these lenses when their contravariant components are exact Bayesian inversions. We prove our main result both abstractly and concretely, for both discrete and continuous states, taking care to illustrate the common structures.
\end{abstract}

\section{Introduction}
\label{sec:org92a0107}
\label{sec:intro}

Bayesian inference appears whenever we wish to understand the latent causes of stochastically generated data. In this paper, we show how the inversion of a generative process fits a pattern known as \emph{optics} \citep{Riley2018Categories,Clarke2020Profunctor} that describes various kinds of bidirectional transformation. In particular, we show that the Bayesian inverse of a complex stochastic process can be constructed compositionally according to this common pattern. We will assume that the reader has some rudimentary knowledge of category theory, but not necessarily either of optics or of categorical approaches to probability. As such, we have attempted to keep this paper self-contained. For basic introductions to category theory and its applications, we recommend \citet{Leinster2016Basic} and \citet{Fong2018Seven}.

As a slogan, our main result is that \emph{Bayesian updates compose optically} (Theorem \ref{thm:optical-bayes}). The following corollary states this more formally:
\begin{cor*}[\ref{cor:optical-bayes}]
Let \(\cat{C}\) be a copy-delete category (Definition \ref{def:cd-cat}), and let \(\cat{C}^\dag\) be its wide subcategory of morphisms that admit Bayesian inversion (Definition \ref{def:admit-bayes}). Then \(\cat{C}^\dag\) embeds functorially into the category of Bayesian lenses \(\Cat{BayesLens}\)  (Definition \ref{def:bayes-lens}).
\end{cor*}
To supply some intuition for these ideas, we situate this work in the nascent discipline of categorical cybernetics: although Bayesian inference is very widely applicable, the cybernetic context makes plain the bidirectional information flow in which we are interested.

We think of a cybernetic system as being embedded in some environment, and aiming to control some aspect of that environment---such as its habitability. In order for the system to achieve its aims, it must maintain a representation of the state of the environment which it seeks to control. But this external state may not be directly accessible to the system, and moreover the process by which the system's inputs are generated from the environmental state is likely to be stochastic; or, the system may only have a partial view of this state. Somehow, in forming a representation of the relevant environmental state, the system must invert this generative process.

In such a setting, we can model the generative process by a stochastic channel \(X \klto Y\), where \(X\) is some space of environmental states and \(Y\) is some space of `sensory' inputs to the system. Part of the system's task is therefore to obtain from this channel an \emph{inverse} channel \(Y \klto X\) by which it can infer, given some sensory data in \(Y\), a belief about the environmental state in \(X\) which caused that sensory data. This process of inference is known as \emph{Bayesian inference}, following Bayes' theorem of probability.

There is an inherent bidirectionality here: a generative forwards channel \(X \klto Y\) and an inverse backwards channel \(Y \klto X\) that, according to Bayes' law, depends also on a prior belief about \(X\). The abstract pattern that captures this kind of `dependent' bidirectionality is called a \emph{lens}. Lenses were originally developed in database theory \citep{Foster2007Combinators}, where the idea is that given a database record and a field within it, one can zoom in on the field and view its value inside the record; then in the other direction, given a record and a new value for a field, one can go back and obtain a correspondingly updated record. Consequently, we call the first transformation \(\mathsf{view}\) and the second \(\mathsf{update}\). The inference process is similar: the environment causes input data (a partial `\(\mathsf{view}\)'); then, given some prior belief about the environment's state and this sensory data, the system can \(\mathsf{update}\) its belief to reflect the new input information.

In the database context, these transformations are functions \(\mathsf{view} : X \to Y\) and \(\mathsf{update} : X \times Y \to X\) that are both morphisms in the same category, typically the category \(\Set\) of sets and functions; such lenses are called \emph{Cartesian}. But, given a stochastic channel \(c : X \klto Y\), the corresponding Bayesian inversion operation is in general \emph{not} another channel \(c^\dag : X \otimes Y \klto X\) in the same category \(\cat{C}\): instead, it is a map of the form \(\Pow X \to \cat{C}(Y, X)\), where \(\Pow X\) denotes some space of states (\emph{i.e.}, prior beliefs) on \(X\). Our first main contribution is to formalize this fibrationally: given a base category of channels \(\cat{C}\), there is a fibre over each \(X : \cat{C}\) whose morphisms \(B \klto A\) are \(X\text{-state-dependent}\) channels of the form \(\Pow X \to \cat{C}(B, A)\), and which compose compatibly with both the horizontal structure in the base category and the vertical structure in the fibres. The operation of Bayesian inversion is such a state-dependent channel, and following \citet{Spivak2019Generalized} we can use this structure to define a particular abstract category of lenses whose \(\mathsf{view}\) maps live in the base and whose \(\mathsf{update}\) maps live in the fibres.

In keeping with the bidirectionality of lenses, the \(\mathsf{view}\) transformations compose covariantly, while the backwards \(\mathsf{update}\) transformations compose contravariantly. This just means that, given a composite generative process \(X \klto Y \klto Z\), we form the composite inverse by first inverting the (causally proximal) second process \(Y \klto Z\), and then inverting the (distal) first process \(X \klto Y\). For example, at the cinema, the projector determines the state of the whole screen, which reflects light onto the retinae of the viewer. Then, while focusing on a small region of the screen, the viewer's brain maintains a belief about the whole screen, and updates it by first inferring from the retinal signals the picture on the region under focus, then inferring the new state of the whole screen from the new belief about this region.

Our second main contribution is thus to show that the Bayesian inversion of a composite channel, following Bayes' rule, is equivalent (up to almost-equality) to the contravariant lens composition of the inversions of the component channels. We prove this result both abstractly and concretely, for both discrete and continuous probability. With the help of the Yoneda embedding, we also show how to translate the fibred category of lenses into optic form, and thereby recover an equivalent category of Bayesian lenses that is indeed Cartesian (in the database sense). This allows us to define \emph{mixed} Bayesian lenses, whose `vertical' category is different from the base category, but which still captures the state-dependence of Bayesian inversion.

Finally, lenses as introduced in the database literature are often accompanied by \emph{lens laws} which capture aspects of their behaviour that are desirable in the database context. We show that `exact' Bayesian lenses are only weakly lawful in this sense, but that this is desirable: the `beliefs' held by a database are Boolean (either true or false), whereas Bayesian beliefs are in general fuzzy mixtures, and such mixing is contradicted by the lens laws.

\paragraph{Overview of paper}
\label{sec:org9b05a2e}

We have attempted to keep the presentation in this paper self-contained, and assume that the reader may not be familiar with either categorical probability theory or coend optics. As such, in section \secref{sec:math-bg}, we summarize the key structures: basic categorical probability in \secref{sec:comp-prob} (abstract and concrete, discrete and continuous) and optics and lenses in \secref{sec:optics}; in Appendix \secref{sec:coends} we give a brief summary of coend calculus necessary for optics. Since a number of our proofs are made graphically, we also introduce the necessary graphical calculi along the way.

In \secref{sec:stat-cat} we introduce fibred categories of state-dependent channels and the corresponding Grothendieck lenses; then in \secref{sec:bayes-lens} we translate these into optic form, defining categories of Bayesian lenses. In \secref{sec:bayes-compose}, we prove that the Bayesian inversion of a composite channel is (almost-)equal to the lens composite of the Bayesian inversions of the channel's factors, in each of the categories introduced in \secref{sec:comp-prob}. Finally, in \secref{sec:lawfulness} we discuss the `lawfulness' of Bayesian lenses.

\paragraph{Contributions}
\label{sec:org9a50efa}

We define a collection of fibred categories whose morphisms depend on states in the base category (Definition \ref{def:stat-cat}), and show that Bayesian inversion is an instance of such a state-dependent morphism (Example \ref{ex:stat-meas}). The abstract pattern is however more general, and so we expect it to be more widely applicable.

We show how to construct categories of optics from such fibred categories, using their (co)Yoneda embeddings to define actegory structures (Proposition \ref{prop:odot-actegory}), and we define a corresponding notion of \emph{Bayesian lens} (Definition \ref{def:bayes-lens}). We generalize this to mixed optics (Definition \ref{def:mixed-bayes-lens}), and exemplify the generalization with state-dependent algebra homomorphisms (Example \ref{ex:stat-alg}).

We prove that the Bayesian inversion of a composite channel coincides with the lens composite of the inversions of its factors (Theorem \ref{thm:optical-bayes}). Consequently, we show that stochastic channels embed functorially into Bayesian lenses (Corollary \ref{cor:optical-bayes}). We show that `exact' Bayesian lenses are only weakly lawful (\secref{sec:lawfulness}).

We hope to have presented these results and constructions pedagogically, so that this paper may serve as a useful introduction to some of the important structures and techniques of the nascent discipline of categorical cybernetics. To this end, the background section (\secref{sec:math-bg}) is comprehensive but informal, and we provide comparative proofs of Theorem \ref{thm:optical-bayes} (abstract and concrete, discrete and continuous).

\paragraph{Notation}
\label{sec:orga476c6e}

We write \(\cat{C}_0\) for the set of objects (0-cells) in the category \(\cat{C}\). We write \(\cat{C}(-, X)\) and \(\cat{C}(X, -)\) for the representable presheaf and copresheaf on the object \(X : \cat{C}\). Where \(\cat{C}\) is supposed to be a category of stochastic channels, we write its composition operator as \(\klcirc\) and denote morphisms (channels) by \(X \klto Y\). Otherwise, we write the composition operator as \(\circ\) and morphisms as \(X \to Y\), except for lenses and optics, where we write \(\lenscirc\) and \((X, A) \lensto (Y, B)\). Given a stochastic channel \(c\), we often adopt `conditional probability notation' \(c(B|x)\) to indicate the probability of \(B\) given \(x\); we remind the reader of this at the relevant points.

\paragraph{Acknowledgements}
\label{sec:org0df9e3f}

The author thanks Bruno Gavranović, Jules Hedges, and Neil Ghani for stimulating and insightful conversations, and credits Jules Hedges for observing the Cartesian lens form of the Bayesian \(\mathsf{update}\) map in discussion at SYCO 6, and for indicating problems with an earlier version of these results.

\section{Mathematical background and graphical calculi}
\label{sec:org23da094}
\label{sec:math-bg}

\subsection{Compositional probability theory}
\label{sec:orgf90a183}
\label{sec:comp-prob}

In informal scientific literature, Bayes' rule is often written in the following form:
\begin{equation*}
\Pr(A|B) = \frac{\Pr(B|A) \cdot \Pr(A)}{\Pr(B)}
\end{equation*}
where \(\Pr(A)\) is the probability of the `event' \(A\), and \(\Pr(A|B)\) is the probability of the event \(A\) given that the event \(B\) occurred; and \emph{vice versa} swapping \(A\) and \(B\). Unfortunately, this notation obscures that there is in general no unique assignment of probabilities to events: different observers can hold different beliefs. Moreover, we are usually less interested in the probability of particular events than in the process of assigning probabilities to arbitrarily chosen beliefs; and what should be done if \(\Pr(B) = 0\) for some \(B\)? The aim in this section is to make this expression suficiently precise for our purposes.

The assignment of probabilities or beliefs to events is formally the task of a \textbf{state} on the space from which the events are drawn; we should think of states as generalizing distributions or measures. We can write \(\Pr_\pi(A)\) to denote the probability of \(A\) \emph{according to the state} \(\pi\). Similarly, we can write \(\Pr_c(B|A)\) to denote the probability of \(B\) given \(A\) according to the \textbf{channel} \(c\), where the channel \(c\) takes events \(A\) as inputs and emits states \(c(A)\) as outputs. This means that we can alternatively write \(\Pr_c(B|A) = \Pr_{c(A)}(B)\). In general, whenever we encounter a `conditional probability distribution', it is formally a stochastic channel.

If the input events are drawn from the space \(X\) and the output states encode beliefs about \(Y\), then the channel \(c\) is of type \(X \klto Y\), written \(c : X \klto Y\). Given a channel \(c : X \klto Y\) and a channel \(d : Y \klto Z\), we can compose them sequentially by marginalizing (averaging) over the possible outcomes in \(Y\), giving a composite channel \(d \klcirc c : X \klto Z\). We will see precisely how this works in various settings below.

Given two spaces \(X\) and \(Y\) of events, we can form beliefs about them jointly, represented by states on the product space denoted \(X \otimes Y\). The numerator in Bayes' rule represents such a joint state, by the law of conditional probability or `product rule':
\begin{equation} \label{eq:pr-cond-prob}
\Pr_\omega(A, B) = \Pr_c(B|A) \cdot \Pr_\pi(A)
\end{equation}
where \(\cdot\) is multiplication of probabilities, \(\pi\) is a state on \(X\), and \(\omega\) denotes the joint state on \(X \otimes Y\). By composing \(c\) and \(\pi\) to form a state \(c \klcirc \pi\) on \(Y\), we can write
\begin{equation*}
\Pr_{\omega'}(B, A) = \Pr_{c^\dag_\pi}(A|B) \cdot \Pr_{c \klcirc \pi}(B)
\end{equation*}
where \(c^\dag_\pi\) will denote the Bayesian inversion of \(c\) with respect to \(\pi\).

Joint states in classical probability theory are symmetric, meaning that there is an isomorphism \(\mathsf{swap} : X \otimes Y \xklto{\sim} Y \otimes X\). Consequently, we have \(\omega' = \mathsf{swap} \klcirc \omega\) and \(\Pr_\omega(A, B) = \Pr_{\omega'}(B, A)\), and thus
\begin{equation} \label{eq:pr-disintegrations}
\Pr_c(B|A) \cdot \Pr_\pi(A) = \Pr_{c^\dag_\pi}(A|B) \cdot \Pr_{c \klcirc \pi}(B)
\end{equation}
where both left- and right-hand sides are called \emph{disintegrations} of \(\omega\) \citep{Cho2017Disintegration}. From this equality, we can write down the usual form of Bayes' theorem, now with the sources of belief indicated:
\begin{equation} \label{eq:pr-bayes}
\Pr_{c^\dag_\pi}(A|B) = \frac{\Pr_c(B|A) \cdot \Pr_\pi(A)}{\Pr_{c \klcirc \pi}(B)} \, .
\end{equation}
As long as \(\Pr_{c \klcirc \pi}(B) \neq 0\), this equality defines the inverse channel \(c^\dag_\pi\). If the division is undefined, or if we cannot guarantee \(\Pr_{c \klcirc \pi}(B) \neq 0\), then \(c^\dag_\pi\) can be any channel satisfying \eqref{eq:pr-disintegrations}.

There is therefore generally no unique Bayesian inversion \(c^\dag : Y \klto X\) for a given channel \(c : X \klto Y\): rather, we have an inverse \(c^\dag_\pi : Y \klto X\) for each prior state \(\pi\) on \(X\); moreover, \(c^\dag_\pi\) is not a ``posterior distribution'' (as written in some literature), but a channel which emits a posterior distribution, given an observation in \(Y\). By allowing \(\pi\) to vary, we obtain a map of the form \(\bdag{c} : \Pow X \to \cat{C}(Y, X)\), where \(\Pow X\) denotes a space of states on \(X\). This is the form described in \secref{sec:intro}, and is the key to the present paper.

\begin{rmk}
There are two easily confused pieces of terminology here. We will call \(c^\dag_\pi := \bdag{c}(\pi)\) the \textbf{Bayesian inversion} of the channel \(c\) with respect to \(\pi\). Then, given some \(y \in Y\), \(c^\dag_\pi (y)\) is a new `posterior' distribution on X. We will call \(c^\dag_\pi(y)\) the \textbf{Bayesian update} of \(\pi\) along \(c\) given \(y\).
\end{rmk}

\subsubsection{Discrete probability}
\label{sec:org6d28d09}
\label{sec:fin-prob}

Interpreting the informal Bayes' rule \eqref{eq:pr-bayes} is simplest in the case of discrete or \emph{finitely-supported} probability, where events are just elements of sets, and a probability distribution is just an assignment of probabilities to these elements such that that sum of all the assignments is \(1\). This situation is formalized by the \emph{finitely-supported distribution monad} \(\Dst : \Set \to \Set\), and in this setting our category of stochastic channels will be its Kleisli category \(\Kl(\Dst)\). Instead of giving a rigorous presentation of this category, we refer the reader to \citet{Fritz2019synthetic,Cho2017Disintegration}, giving alternatively a self-contained introduction of the structures relevant for our purposes.

The functor \(\Dst : \Set \to \Set\) acts on a set \(X\) by returning the set \(\Dst X\) of finite probability distributions over \(X\): that is, the set of functions \(p : X \to [0, 1]\) such that \(p(x) \neq 0\) for only finitely many elements \(x \in X\), and \(\sum_{x:X} p(x) = 1\). We can think of \(\Dst X\) as a (convex) vector space, with basis vectors \(\ket{x}\) given by the elements \(x\) of \(X\). We can then write a (finitely-supported) distribution \(p : X \to [0,1]\) as a convex (weighted) sum of these basis vectors \(\sum_{x:X} \,\boxed{\,p(x)} \, \ket{x}\), where the expression inside the box evaluates to a probability.

\paragraph{Channels in \(\Kl(\Dst)\): stochastic matrices}
\label{sec:org9d2070b}

The objects of \(\Kl(\Dst)\) are sets, and morphisms \(X \klto Y\) are functions \(X \to \Dst Y\); equivalently, using the Cartesian-closed structure of \(\Set\), they are functions \(X \times Y \to [0, 1]\), \emph{i.e.} (left stochastic) matrices of size \(|X| \times |Y|\), each of whose columns sums to \(1\). We think of morphisms as \emph{stochastic channels} emitting outputs stochastically for each input, with the stochasticity encoded in the output states. We adopt `conditional probability' notation: given \(p : X \klto Y\), \(x \in X\) and \(y \in Y\), we write \(p(y|x) := p(x)(y) \in [0, 1]\) for ``the probabilty of \(y\) \emph{given} \(x\), according to \(p\)''.

Identity morphisms \(\id_X : X \klto X\) in \(\Kl(\Dst)\) take points to `Dirac distributions': \(\id_X := x \mapsto 1 \ket{x}\); these are the unit maps \(\eta_X\) of the monad structure on \(\Dst\). Note that any function \(f : Y \to X\) can be made into a (deterministic) channel \(\langle f \rangle = \eta_X \circ f : Y \to \Dst X\) by post-composition with \(\eta_X\).

Given \(p : X \to \Dst Y\) and \(q : Y \to \Dst Z\), we write their (sequential) composite as \(q \klcirc p : X \to \Dst Z\), constructed by `averaging over' or `marginalizing out' \(Y\) via the Chapman-Kolmogorov equation:
\begin{equation*}
q \klcirc p : X \to \Dst Z := x \mapsto \sum_{z : Z} \, \boxed{\sum_{y : Y} q(z|y) \cdot p(y|x)} \, \ket{z} .
\end{equation*}
Note that this is just the (matrix) product of the stochastic matrices corresponding to the channels \(q\) and \(p\). Abstractly, the composite is formed as the composite \(q^\rhd \circ p\) in \(\Set\) of \(p\) followed by the \emph{Kleisli extension} \(q^\rhd : \Dst Y \to \Dst Z\) of \(q\). Kleisli extension \((-)^\rhd\) turns any morphism \(q : Y \to \Dst Z\) into a morphism \(q^\rhd : \Dst Y \to \Dst Z\), and in \(\Kl(\Dst)\), it is defined using marginalization as follows:
\begin{equation} \label{eq:klext-finite}
q^\rhd : \Dst Y \to \Dst Z := \rho \mapsto \sum_{z:Z} \, \boxed{\sum_{y:Y} q(z|y) \cdot \rho(y)} \, \ket{z} .
\end{equation}

\paragraph{Monoidal structure: joint states and parallel channels}
\label{sec:org228c9b1}

\(\Dst\) is a \emph{monoidal monad}, meaning that there is a family of maps \(\Dst X \times \Dst Y \to \Dst(X \times Y)\), natural in \(X\) and \(Y\), which take a pair of distributions \((\rho, \sigma)\) in \(\Dst X \times \Dst Y\) to the joint distribution on \(X \times Y\) given by \((x, y) \mapsto \rho(x) \cdot \sigma(y)\); \(\rho\) and \(\sigma\) are then the (independent) marginals of this joint distribution. This structure makes \(\Kl(\Dst)\) into a monoidal category, with a tensor product functor \(\otimes : \Kl(\Dst) \times \Kl(\Dst) \to \Kl(\Dst)\). This functor is defined on pairs of objects \(X\) and \(Y\) as their product \(X \otimes Y = X \times Y\), and on stochastic maps \(f : X \to \Dst A\) and \(g : Y \to \Dst B\) as the `parallel composite' \(f \otimes g : X \times Y \to \Dst(A \times B)\) via the monoidal structure of \(\Dst\):
\begin{equation*}
X \times Y \xto{f \times g} \Dst A \times \Dst B \rightarrow \Dst(A \times B) \, .
\end{equation*}
Note that because not all joint states have independent marginals, the monoidal product \(\otimes\) is not Cartesian: that is, given an arbitrary \(\omega : \Dst (X \otimes Y)\), we do not have \(\omega \cong (\rho, \sigma)\) for some \(\rho : \Dst X\) and \(\sigma : \Dst Y\).

As indicated in \secref{sec:comp-prob}, \(\Kl(\Dst)\) is \emph{symmetric} monoidal: since \(X \times Y \cong Y \times X\), there are natural `swap' isomorphisms \(\mathsf{swap}_{X,Y} : X \otimes Y \xklto{\sim} Y \otimes X\) and \(\mathsf{swap}_{Y,X} : Y \otimes X \xklto{\sim} X \otimes Y\) such that \(\mathsf{swap}_{Y,X} \klcirc \mathsf{swap}_{X,Y} = \id_{X \otimes Y}\).

The tensor product \(\otimes\) is equipped with a \emph{unit} object \(I\), which means that, for all objects \(X\), there are natural isomorphisms \(\lambda_X : I \otimes X \xklto{\sim} X\) and \(\rho_X : X \otimes I \xklto{\sim} X\) called the \emph{left and right unitors}. Note that, when \(\Set\) is equipped with the Cartesian product \(\times\), the monoidal unit \(I\) is the singleton set \(1 = \{\ast\}\), and we have \(1 \times X \overset{\lambda}{\cong} X \overset{\rho}{\cong} 1 \times X\). Since \(\otimes\) on \(\Kl(\Dst)\) derives from \(\times\) on \(\Set\), we also have \(I = 1\) in \(\Kl(\Dst)\). Finally, note that states \(\pi : \Dst X\) correspond isomorphically to functions \(\pi : 1 \to \Dst X\), and hence channels \(\pi : I \klto X\).

\paragraph{Marginalization: discarding, causality, and projections}
\label{sec:org6938bf9}

Given a joint distribution \(\omega : 1 \to \Dst(X \times Y)\), we can recover each marginal \(\omega_1 : 1 \to \Dst X\) or \(\omega_2 : 1 \to \Dst Y\) by marginalizing out the other. Categorically, this is captured by the existence of \emph{discarding} maps \(\ground_X : X \to \Dst 1 \cong 1 := x \mapsto 1 \ket{\ast}\). From the discarding maps, we can construct \emph{projection} maps for the tensor product; these witness marginalization:
\[ \pi_1 : X \times Y \to \Dst X := X \times Y \xto{\id \times \ground} \Dst X \times 1 \cong \Dst X \]
and
\[ \pi_2 : X \times Y \to \Dst Y := X \times Y \xto{\ground \times \id} 1 \times \Dst Y \cong \Dst Y \, , \]
which are natural in that \(\pi_i \bullet (f_1 \otimes f_2) = f_i \bullet \pi_i\). Explicitly, using the definitions of \(\id\) and \(\ground\) given above, we have \(\pi_1(x, y) = 1 \ket{x}\) and \(\pi_2(x, y) = 1 \ket{y}\); and so, given some joint distribution \(\omega : 1 \to \Dst(X \times Y)\), \(\omega_1 = \pi_1 \bullet \omega = \sum_{y:Y} \boxed{\omega(x, y)} \ket{x}\), and similarly, \(\omega_2 = \pi_2 \bullet \omega = \sum_{x:X} \boxed{\omega(x, y)} \ket{y}\).

We say that a stochastic map \(f\) is \emph{causal} if doing \(f\) then throwing away the output is the same as just throwing away the input: \(\ground \klcirc f = \ground\); this means that \(f\) cannot affect states `in its past'. In \(\Kl(\Dst)\), every map is causal (the discarding maps are natural), but this will not be true in all the categories of interest to us in this paper.

\paragraph{Copying}
\label{sec:org2c5f76b}
\label{sec:finite-copying}

In order to define lens composition, we need one more piece of structure: a family of copying maps, denoted \(\bcopier\). In \(\Kl(\Dst)\), these are the maps \(\bcopier_X : X \to \Dst(X \times X) := x \mapsto 1 \ket{x, x}\). Together with the discarding maps \(\ground_X\), they make every object \(X\) into a commutative comonoid; this will be elaborated further in \secref{sec:graph-calc}. Note that the copying maps are not natural in \(\Kl(\Dst)\): in general, \(\bcopier \bullet f \neq f \otimes f \bullet \bcopier\). Those maps \(f\) that do satisfy this equality are \emph{comonoid homomorphisms}, and in \(\Kl(\Dst)\) correspond to the deterministic maps (\emph{i.e.} those that emit Dirac delta distributions).

\paragraph{Bayesian updating}
\label{sec:org0fed4e8}

We can now instantiate Bayesian updating in \(\Kl(\Dst)\). Given a channel \(p : X \to \Dst Y\) and a prior \(\rho : 1 \to \Dst X\), the Bayesian update of \(\rho\) along \(p\) is given by the function
\begin{equation} \label{eq:finite-bayes}
\bdag{p} : \Dst X \times Y \to \Dst X := \rho \times y \mapsto \sum_{x : X} \, \boxed{\frac{p(y|x) \cdot \rho(x)}{\sum_{x':X} p(y|x') \cdot \rho(x')}} \, \ket{x} = \sum_{x : X} \, \boxed{\frac{p(y|x) \cdot \rho(x)}{[p \klcirc \rho](y)}} \, \ket{x} \, .
\end{equation}
The expression on the right-hand side is easily seen to correspond to the informal expression of Bayes' rule in equation \eqref{eq:pr-bayes}.

\subsubsection{Graphical calculus}
\label{sec:orgea1a4e3}
\label{sec:graph-calc}

We now move from \(\Kl(\Dst)\) to a more general setting. We will assume that, in each category \(\cat{C}\) of stochastic channels of interest to us, we are able to form parallel channels and coherently copy and delete states, analogously to the discrete case in \secref{sec:fin-prob}. This means that \(\cat{C}\) must be a \emph{copy-delete category} \citep{Cho2017Disintegration}.

\begin{defn}[{\textcite[Def. 2.2]{Cho2017Disintegration}}] \label{def:cd-cat}
A \textbf{copy-delete category} is a symmetric monoidal category \((\cat{C}, \otimes, I)\) in which every object \(X\) is supplied with a commutative comonoid structure \((\bcopier_X, \ground_X)\) compatible with the monoidal structure of \((\otimes, I)\). An \textbf{affine} copy-delete category, or \textbf{Markov category} \citep{Fritz2019synthetic}, is a copy-delete category in which every channel \(c\) is causal in the sense that \(\ground \klcirc c = \ground\). Equivalently, a Markov category is a copy-delete category in which the monoidal unit \(I\) is the terminal object.
\end{defn}

Symmetric monoidal categories, and (co)monoids within them, admit a formal graphical calculus that substantially simplifies many calculations involving complex morphisms: proofs of many equalities reduce to visual demonstrations of isotopy, and structural morphisms such as the symmetry of the monoidal product acquire intuitive topological depictions. We make substantial use of this calculus below, and summarize its features here. For more details, see \(\text{\textcite[{\S}2]{Cho2017Disintegration} or \textcite[{\S}2]{Fritz2019synthetic}}\) or the references cited therein.

\paragraph{Basic structure}
\label{sec:orge0de2c3}

Diagrams in the graphical calculus represent morphisms. We draw morphisms as boxes on strings, labelling the strings with the corresponding objects in the category. Identity morphisms are drawn as plain strings. Sequential composition is represented by connecting strings together; and parallel composition \(\otimes\) by placing diagrams adjacent to one another.

Diagrams for \(\cat{C}\) will be read vertically, with information flowing upwards (from bottom to top). This way, \(c : X \klto Y\), \(\id_X : X \klto X\), \(d \klcirc c : X \xklto{c} Y \xklto{d} Z\), and \(f \otimes g : X \otimes Y \klto A \otimes B\) are depicted respectively as:
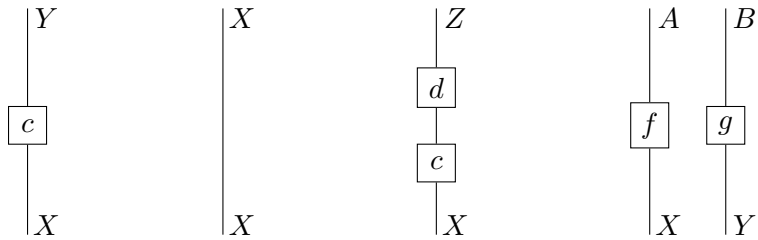
\[
\hspace{0.125\linewidth} \input{img/channel-c.tikz}
\hspace{0.125\linewidth} \begin{tikzpicture}
	\begin{pgfonlayer}{nodelayer}
		\node [style=none] (0) at (0, -3) {};
		\node [style=none] (2) at (0, 3) {};
		\node [style=none] (3) at (0.5, -2.75) {$X$};
		\node [style=none] (4) at (0.5, 2.75) {$X$};
	\end{pgfonlayer}
	\begin{pgfonlayer}{edgelayer}
		\draw (0.center) to (2.center);
	\end{pgfonlayer}
\end{tikzpicture}

\hspace{0.125\linewidth} \input{img/channel-dc.tikz}
\hspace{0.125\linewidth} \input{img/channel-f_g.tikz}
\hspace{0.125\linewidth}
\]
We represent (the identity morphism on) the monoidal unit \(I\) as an empty diagram: that is, we leave it implicit in the graphical representation.

\paragraph{States and effects}
\label{sec:orga2b7153}

In \(\Kl(\Dst)\) we saw that a channel \(I \klto X\) was a finitely supported distribution over \(X\). In general, we will call such a morphism \(I \klto X\) a \emph{state} of \(X\). Dually, a morphism \(X \klto I\) in \(\cat{C}\) is called an \emph{effect}. States \(\sigma : I \klto X\) and effects \(\eta : X \klto I\) will be represented as follows:
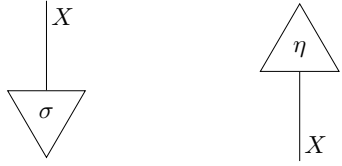
\begin{gather*}
\scalebox{0.85}{\begin{tikzpicture}
	\begin{pgfonlayer}{nodelayer}
		\node [style=state180] (0) at (0, -2) {$\sigma$};
		\node [style=none] (1) at (0, 1.5) {};
		\node [style=none] (2) at (0.5, 1) {$X$};
	\end{pgfonlayer}
	\begin{pgfonlayer}{edgelayer}
		\draw (0) to (1.center);
	\end{pgfonlayer}
\end{tikzpicture}
}
\hspace{0.125\linewidth}
\scalebox{0.85}{\begin{tikzpicture}
	\begin{pgfonlayer}{nodelayer}
		\node [style=state0] (0) at (0, 0) {$\eta$};
		\node [style=none] (1) at (0, -3.5) {};
		\node [style=none] (2) at (0.5, -3) {$X$};
	\end{pgfonlayer}
	\begin{pgfonlayer}{edgelayer}
		\draw (0) to (1.center);
	\end{pgfonlayer}
\end{tikzpicture}
}
\end{gather*}

\paragraph{Discarding, causality, marginalization and projections}
\label{sec:orgefcceac}

As noted in \secref{sec:fin-prob}, in \(\Kl(\Dst)\) there is only one possible effect of each type \(X\), given by the discarding map \(\ground_X : X \klto I\). This uniqueness follows categorically from the fact that the object \(I = 1\) is the terminal object in \(\Kl(\Dst)\) --- meaning that there is a unique map from every object into \(I\) --- and is equivalent to the condition that every channel \(c : X \klto Y\) is causal:
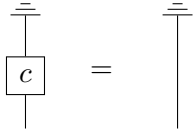
\[
\input{img/causality-condition.tikz}
\]

From the discarding maps, we constructed projections in \(\Kl(\Dst)\) witnessing the marginalization of joint states. This has a pleasing graphical representation. Suppose a joint state \(\omega : I \klto X \otimes Y\) has marginals \(\omega_1 : I \klto X\) and \(\omega_2 : I \klto Y\). Then
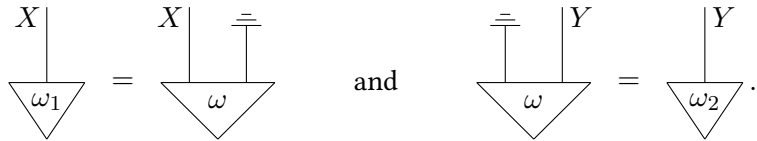
\[
\input{img/marginalization-X.tikz}
\hspace{0.06\linewidth}
\text{and}
\hspace{0.06\linewidth}
\input{img/marginalization-Y.tikz}
\, .
\]

\paragraph{Copying}
\label{sec:orgf9847c6}

The copying maps \(\bcopier_X : X \klto X \otimes X\) have a similarly intuitive graphical representation. They are required to interact nicely with the discarding maps, making each object \(X\) into a comonoid (satisfying unitality and associativity):
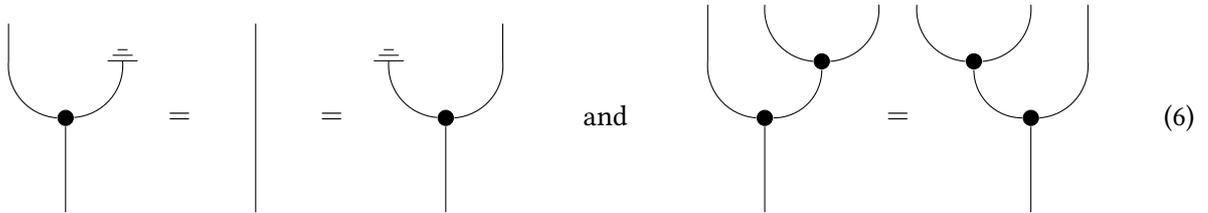
\begin{equation} \label{eq:comonoid-law}
\input{img/copy-delete-identity.tikz}
\hspace{0.06\linewidth}
\text{and}
\hspace{0.06\linewidth}
\input{img/copy-copy-identity.tikz}
\end{equation}
A category with such comonoid structure \((\bcopier_X, \ground_X)\) for every object \(X\) is said to \emph{supply comonoids} \citep{Fong2019Supplying}.

We will draw the \(\mathsf{swap}\) isomorphisms of the symmetric monoidal structure as the swapping of wires, and assume that the copying maps commute with this swapping, making the comonoids into \emph{commutative} comonoids:
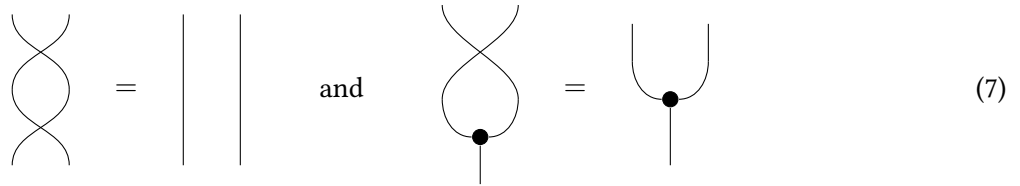
\begin{equation} \label{eq:comonoid-commute}
\input{img/swap-swap-identity.tikz}
\hspace{0.06\linewidth}
\text{and}
\hspace{0.06\linewidth}
\input{img/copy-swap-identity.tikz}
\end{equation}

\paragraph{Conditional probability}
\label{sec:orgfccf9c7}

We end this summary with a graphical statement of the law of conditional probability \eqref{eq:pr-cond-prob}. Suppose as before that \(A \subseteq X\) and \(B \subseteq Y\), with \(\omega : 1 \klto X \otimes Y\), \(c : X \klto Y\), and \(\pi : 1 \klto X\). The disintegration \(\Pr_\omega(A, B) = \Pr_c(B|A) \cdot \Pr_\pi(A)\) then takes the graphical form
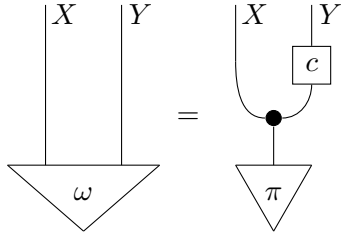
\[
\input{img/disintegration-c-pi.tikz}
\]
with the marginals \(\pi\) and \(c \klcirc \pi\) of \(\omega\) given by
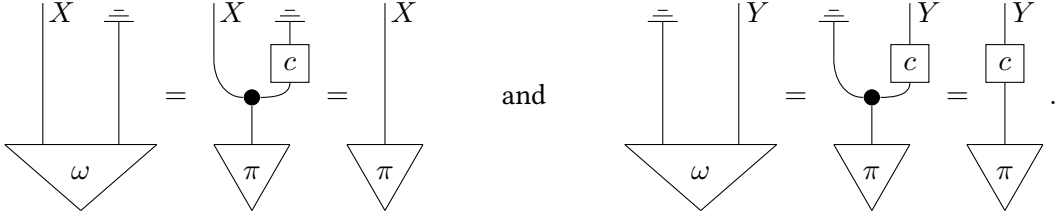
\[
\input{img/disintegration-marginal-X.tikz}
\hspace{0.06\linewidth}
\text{and}
\hspace{0.06\linewidth}
\input{img/disintegration-marginal-Y.tikz} \, .
\]

\subsubsection{Abstract Bayesian inversion}
\label{sec:orgf7fad60}
\label{sec:abstract-bayes}

Bayesian inversion informally satisfies the equation \(\Pr_c(B|A) \cdot \Pr_\pi(A) = \Pr_{c^\dag_\pi}(A|B) \cdot \Pr_{c \klcirc \pi}(B)\) \eqref{eq:pr-disintegrations}. Given the structures introduced above, we can formalize this rule, depicting it as the following graphical equality \parencite[eq. 5]{Cho2017Disintegration}:
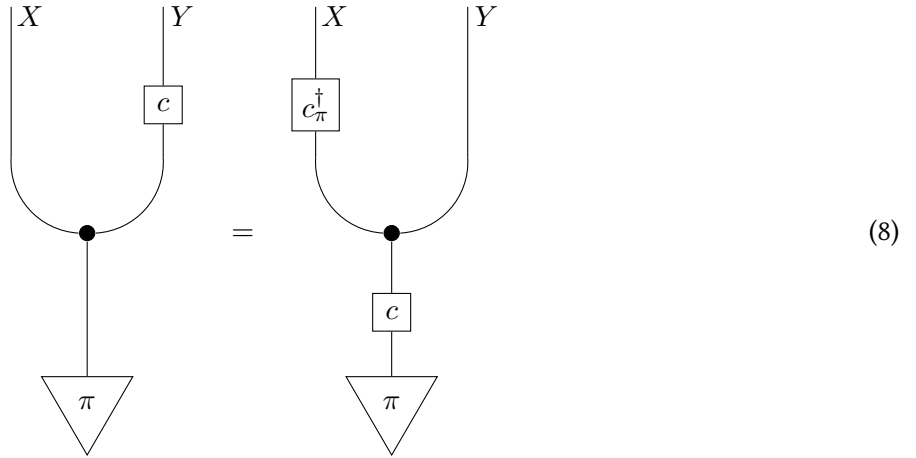
\begin{equation} \label{eq:bayes-abstr}
\input{img/joint-c-pi.tikz} \quad = \quad \input{img/joint-cdag-c-pi.tikz}
\end{equation}
This diagram can be interpreted as follows. Given a prior \(\pi : I \klto X\) and a channel \(c : X \klto Y\), we form the joint distribution \(\omega := (\id_X \otimes \, c) \klcirc \bcopier_X \klcirc \pi : I \klto X \otimes Y\) shown on the left hand side: this is the product rule form, \(\Pr_\omega(A, B) = \Pr_c(B | A) \cdot \Pr_\pi(A)\), and \(\pi\) is the corresponding \(X\text{-marginal}\). As in the concrete case of \(\Kl(\Dst)\), we seek an inverse channel \(Y \klto X\) witnessing the `dual' form of the rule, \(\Pr_\omega(A, B) = \Pr(A | B) \cdot \Pr(B)\); this is depicted on the right hand side. By discarding \(X\), we see that \(c \klcirc \pi : I \klto Y\) is the \(Y\text{-marginal}\) witnessing \(\Pr(B)\). So any channel \(c^\dag_\pi : Y \klto X\) witnessing \(\Pr(A | B)\) and satisfying the equality above is a Bayesian inverse of \(c\) with respect to \(\pi\).

\begin{defn} \label{def:admit-bayes}
We say that a channel \(c : X \klto Y\) \textbf{admits Bayesian inversion} with respect to \(\pi : I \klto X\) if there exists a channel \(c^\dag_\pi : Y \klto X\) satisfying equation \eqref{eq:bayes-abstr}. We say that \(c\) admits Bayesian inversion \emph{tout court} if \(c\) admits Bayesian inversion with respect to all states \(\pi : I \klto X\) such that \(c \klcirc \pi\) has non-empty support.
\end{defn}

\subsubsection{Density functions}
\label{sec:org1a38b55}
\label{sec:density-functions}

Abstract Bayesian inversion \eqref{eq:bayes-abstr} generalizes the product rule form of Bayes' theorem \eqref{eq:pr-disintegrations} but in most applications, we are interested in a specific channel witnessing \(\Pr(A | B) = \Pr(B | A) \cdot \Pr(A) / \Pr(B)\). In the common setting of continuous spaces, this is often written informally as
\begin{equation} \label{eq:bayes-density-informal}
p(x|y) 
= \frac{p(y|x) \cdot p(x)}{p(y)}
= \frac{p(y|x) \cdot p(x)}{\int_{x':X} p(y|x') \cdot p(x') \, \d x'}
\end{equation}
but the formal semantics of such an expression are not trivial: for instance, what is the object \(p(y|x)\), and how does it relate to a channel \(c : X \klto Y\)? Moreover, it is not generally true that, given a channel \(c : X \klto Y\) and prior \(\pi : I \klto X\), a Bayesian inversion \(c^\dag_\pi : Y \klto X\) necessarily exists \citep{Stoyanov2014Counterexamples}!

We can interpret \(p(y|x)\) as a \emph{density function} for a channel: an effect \(X \otimes Y \klto I\) in our ambient category \(\cat{C}\). Consequently, \(\cat{C}\) cannot be semicartesian (\emph{i.e.}, \(\cat{C}\) cannot be an affine copy-delete category)---as this would trivialize all density functions---though it must still supply comonoids. We can think of this as expanding the collection of channels in the category to include acausal or `partial' maps and unnormalized distributions or states. An example of such a category is \(\Kl(\Dst_{\leq 1})\), whose objects are sets (as for \(\Kl(\Dst)\)), and whose morphisms \(X \klto Y\) are functions \(X \to \Dst(Y + 1)\), where \(Y + 1\) is the disjoint union of \(Y\) with \(1 = \{\ast\}\). Then a stochastic map is partial if it sends any probability to the added element \(\ast\). The subcategory of `total' (equivalently, causal) maps is \(\Kl(\Dst)\) \citep{Cho2015Introduction}.

\begin{defn}[Density functions] \label{def:density-functions}
A channel \(c : X \klto Y\) is said to be \textbf{represented by an effect} \(p : X \otimes Y \klto I\) with respect to \(\mu : I \klto Y\) if
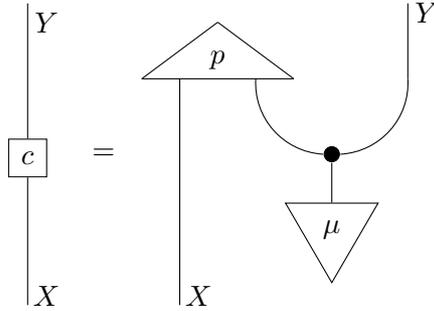
\[ \input{img/def-density-function-c.tikz}. \]
In this case, we call \(p\) a \textbf{density function} for \(c\).
\end{defn}

We will also need the concepts of almost-equality and almost-invertibility.

\begin{defn}[Almost-equality, almost-invertibility] \label{def:almost-eq}
Given a state \(\pi : I \klto X\), we say that two channels \(c : X \klto Y\) and \(d : X \klto Y\) are \(\mathbf{\pi}\textbf{-almost-equal}\), denoted \(c \overset{\pi}{\sim} d\), if
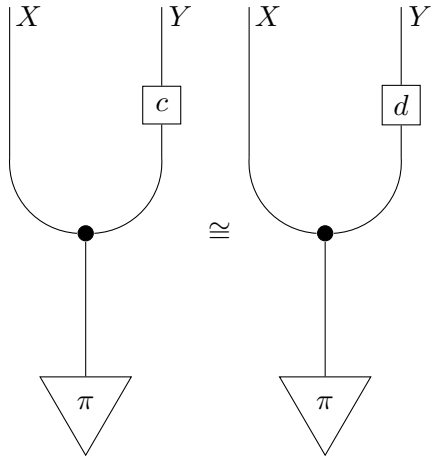
\[ \input{img/joint-c-pi.tikz} \ \cong\ \input{img/joint-d-pi.tikz} \]
and we say that an effect \(p : X \klto I\) is \(\mathbf{\pi}\textbf{-almost-invertible}\) with \(\mathbf{\pi}\textbf{-almost-inverse } q : X \klto I\) if
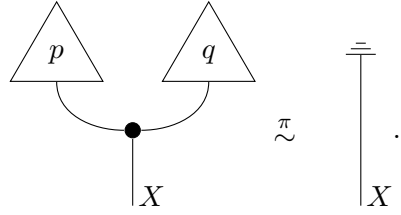
\[ \input{img/almost-invertibility.tikz}. \]
\end{defn}

\begin{prop}[Composition preserves almost-equality] \label{prop:comp-preserve-almost-eq}
If \(c \overset{\pi}{\sim} d\), then \(f \klcirc c \overset{\pi}{\sim} f \klcirc d\).
\begin{proof}
Immediate from the definition of almost-equality.
\end{proof}
\end{prop}

\begin{prop}[Almost-inverses are almost-equal] \label{prop:almost-inverse-almost-equal}
Suppose \(q : X \klto I\) and \(r : X \klto I\) are both \(\pi\text{-almost-inverses}\) for the effect \(p : X \klto I\). Then \(q \overset{\pi}{\sim} r\).
\begin{proof}
Deferred to Appendix \secref{sec:proof:almost-inverse-almost-equal}.
\end{proof}
\end{prop}

With these notions, we can characterise Bayesian inversion via density functions.

\begin{prop}[Bayesian inversion via density functions; \citet{Cho2017Disintegration}] \label{prop:bayes-density-graph}
Suppose \(c : X \klto Y\) is represented by the effect \(p\) with respect to \(\mu\). The Bayesian inverse \(c^\dag_\pi : Y \klto X\) of \(c\) with respect to \(\pi : I \klto X\) is given by
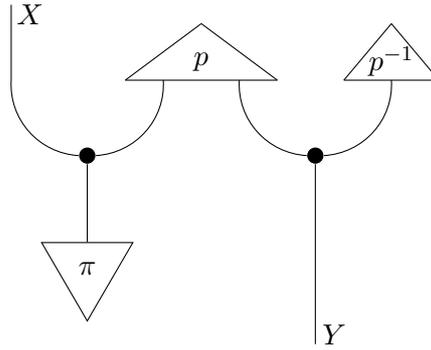
\[ \input{img/channel-density-cdag-pi.tikz} \]
where \(p^{-1} : Y \klto I\) is a \(\mu\text{-almost-inverse}\) for the effect
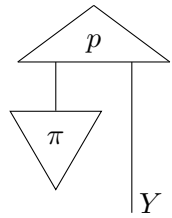
\[ \input{img/likelihood-p-pi.tikz} \]
\begin{proof}
Deferred to Appendix \secref{sec:proof:bayes-density-graph}.
\end{proof}
\end{prop}

The following proposition is an immediate consequence of the definition of almost-equality and of the abstract characterisation of Bayesian inversion \eqref{eq:bayes-abstr}. We omit the proof.
\begin{prop}[Bayesian inverses are almost-equal] \label{prop:bayes-almost-equal}
Suppose \(\alpha : Y \klto X\) and \(\beta : Y \klto X\) are both Bayesian inversions of the channel \(c : X \klto Y\) with respect to \(\pi : I \klto X\). Then \(\alpha \overset{c \klcirc \pi}{\sim} \beta\).
\end{prop}

We will also need the following two technical results about almost-equality.

\begin{lemma} \label{lemma:eff-chan-almost-eq}
Suppose the channels \(\alpha\) and \(\beta\) satisfy the following relations for some \(f,q,r\):
\[
\input{img/eff-chan-almost-eq-1.tikz}
\hspace{0.06\linewidth}
\text{and}
\hspace{0.06\linewidth}
\input{img/eff-chan-almost-eq-2.tikz}
\]
Suppose \(q \overset{\mu}{\sim} r\). Then \(\alpha \overset{\mu}{\sim} \beta\).
\begin{proof}
Deferred to Appendix \secref{sec:proof:eff-chan-almost-eq}.
\end{proof}
\end{lemma}

\begin{lemma} \label{lemma:eff-blocks-almost-eq}
If the channel \(d\) is represented by an effect with respect to the state \(\nu\), and if \(f \overset{\nu}{\sim} g\), then \(f \overset{d \klcirc \rho}{\sim} g\) for any state \(\rho\) on the domain of \(d\).
\begin{proof}
Deferred to Appendix \secref{sec:proof:eff-blocks-almost-eq}.
\end{proof}
\end{lemma}

\subsubsection{S-finite kernels}
\label{sec:orge9cae7b}
\label{sec:sfKrn}

To represent channels by concrete effects (\emph{i.e.}, density functions), we work in the category  \(\Cat{sfKrn}\) of measurable spaces and s-finite kernels. Once again, we only sketch the structure of this category, and refer the reader to \citet{Cho2017Disintegration,Staton2017Commutative} for elaboration.

Objects in \(\Cat{sfKrn}\) are measurable spaces \((X, \Sigma_X)\); often we will just write \(X\), and leave the \(\sigma\text{-algebra } \Sigma_X\) implicit. Morphisms \((X, \Sigma_X) \klto (Y, \Sigma_Y)\) are s-finite kernels. A \emph{kernel} \(k\) from \(X\) to \(Y\) is a function \(k : X \times \Sigma_Y \to [0, \infty]\) satisfying the following conditions:
\begin{itemize}
\item for all \(x \in X\), \(k(x, -) : \Sigma_Y \to [0, \infty]\) is a measure; and
\item for all \(B \in \Sigma_Y\), \(k(-, B) : X \to [0, \infty]\) is measurable.
\end{itemize}
A kernel \(k : X \times \Sigma_Y \to [0, \infty]\) is \emph{finite} if there exists some \(r \in [0, \infty)\) such that, for all \(x \in X\), \(k(x, Y) \leq r\). And \(k\) is \emph{s-finite} if it is the sum of at most countably many finite kernels \(k_n\), \(k = \sum_{n : \nn} k_n\).

Identity morphisms \(\id_X : X \klto X\) are Dirac kernels \(\delta_X : X \times \Sigma_X \to [0, \infty] := x \times A \mapsto 1\) iff \(x \in A\) and 0 otherwise. Composition is given by a Chapman-Kolmogorov equation, analogously to composition in \(\Kl(\Dst)\). Suppose \(c : X \klto Y\) and \(d : Y \klto Z\). Then
\[
d \klcirc c : X \times \Sigma_Z \to [0, \infty]
:= x \times C \mapsto \int_{y:Y} d(C|y) \, c(\d y| x)
\]
where we have again used the `conditional probability' notation \(d(C|y) := d \circ (y \times C)\). Reading \(d(C|y)\) from left to right, we can think of this notation as akin to reading the string diagrams from top to bottom, \emph{i.e.} from output(s) to input(s).

\paragraph{Monoidal structure on \(\Cat{sfKrn}\)}
\label{sec:orga8d0b6b}

There is a monoidal structure on \(\Cat{sfKrn}\) analogous to that on \(\Kl(\Dst)\). On objects, \(X \otimes Y\) is the Cartesian product \(X \times Y\) of measurable spaces. On morphisms, \(f \otimes g : X \otimes Y \klto A \otimes B\) is given by
\[
f \otimes g : (X \times Y) \times \Sigma_{A \times B}
:= (x \times y) \times E \mapsto \int_{a:A} \int_{b:B} \delta_{A \otimes B}(E|x, y) \, f(\d a|x) \, g(\d b|y)
\]
where, as above, \(\delta_{A \otimes B}(E|a, b) = 1\) iff \((a, b) \in E\) and 0 otherwise. Note that \((f \otimes g)(E|x, y) = (g \otimes f)(E|y, x)\) for all s-finite kernels (and all \(E\), \(x\) and \(y\)), by the Fubini-Tonelli theorem for s-finite measures \citep{Cho2017Disintegration,Staton2017Commutative}, and so \(\otimes\) is symmetric on \(\Cat{sfKrn}\).

The monoidal unit in \(\Cat{sfKrn}\) is again \(I = 1\), the singleton set. Unlike in \(\Kl(\Dst)\), however, we do have nontrivial effects \(p : X \klto I\), given by kernels \(p : (X \times \Sigma_1) \cong X \to [0, \infty]\), with which we will represent density functions.

\paragraph{Comonoids in \(\Cat{sfKrn}\)}
\label{sec:org1743d45}

\(\Cat{sfKrn}\) also supplies comonoids, again analogous to those in \(\Kl(\Dst)\). Discarding is given by the family of effects \(\ground_X : X \to [0, \infty] := x \mapsto 1\), and copying is again Dirac-like: \(\bcopier_X : X \times \Sigma_{X \times X} := x \times E \mapsto 1\) iff \((x, x) \in E\) and 0 otherwise. Because we have nontrivial effects, discarding is only natural for causal or `total' channels: if \(c\) satisfies \(\ground \klcirc c = \ground\), then \(c(-|x)\) is a probability measure for all \(x\) in the domain\footnote{This means that \(\Kl(\Giry)\) is the subcategory of total maps in \(\Cat{sfKrn}\), where \(\Giry\) is the \emph{Giry monad} taking each measurable space \(X\) to the space \(\Giry X\) of measures over \(X\).}. And, once again, copying is natural (that is, \(\bcopier \klcirc c = (c \otimes c) \klcirc \bcopier\)) iff the channel is deterministic.

\paragraph{Channels represented by effects}
\label{sec:org01f097a}

We can interpret the string diagrams of \secref{sec:graph-calc} in \(\Cat{sfKrn}\), and we will do so by following the intuition of the conditional probability notation and reading the string diagrams from outputs to inputs. Hence, if \(c : X \klto Y\) is represented by the effect \(p : X \otimes Y \klto I\) with respect to the measure \(\mu : I \klto Y\), then
\[
c : X \times \Sigma_Y \to [0, \infty]
:= x \times B \mapsto \int_{y:B} \mu(\d y) \, p(y | x) .
\]
Note that we also use conditional probability notation for density functions, and so \(p(y|x) := p \circ (x \times y)\).

Suppose that \(c : X \klto Y\) is indeed represented by \(p\) with respect to \(\mu\), and that \(d : Y \klto Z\) is represented by \(q : Y \otimes Z \klto I\) with respect to \(\nu : I \klto Z\). Then in \(\Cat{sfKrn}\), \(d \klcirc c : X \klto Z\) is given by
\[
d \klcirc c : X \times \Sigma_Z 
:= x \times C \mapsto \int_{z:C} \nu(\d z) \, \int_{y:Y} q(z|y) \, \mu(\d y) \, p(y|x)
\]
Alternatively, by defining the effect \((p \mu q) : X \otimes Z \klto I\) as
\[
(p \mu q) : X \times Z \to [0, \infty]
:= x \times z \mapsto \int_{y:Y} q(z|y) \, \mu(\d y) \, p(y|x),
\]
we can write \(d \klcirc c\) as
\[
d \klcirc c : X \times \Sigma_Z
:= x \times C \mapsto \int_{z:C} \nu(\d z) \, (p \mu q)(z|x) .
\]

\paragraph{Bayesian inversion via density functions}
\label{sec:org66802f4}

Once again writing \(\pi : I \klto X\) for a prior on X, and interpreting the string diagram of Proposition \ref{prop:bayes-density-graph} for \(c^\dag_\pi : Y \klto X\) in \(\Cat{sfKrn}\), we have
\begin{equation} \label{eq:bayes-density-krn}
\begin{aligned}
c^\dag_\pi : Y \times \Sigma_X \to [0, \infty]
:= y \times A & \mapsto \left( \int_{x:A} \pi(\d x) \, p(y|x) \right) p^{-1}(y) \\
&= p^{-1}(y) \int_{x:A} p(y|x) \, \pi(\d x) ,
\end{aligned}
\end{equation}
where \(p^{-1} : Y \klto I\) is a \(\mu\text{-almost-inverse}\) for effect \(p \klcirc (\pi \otimes \id_Y)\), and is given up to \(\mu\text{-almost-equality}\) by
\[
p^{-1} : Y \to [0, \infty]
:= y \mapsto \left( \int_{x:X} p(y|x) \, \pi(\d x) \right)^{-1} \, .
\]
Note that from this we recover the informal form of Bayes' rule for measurable spaces \eqref{eq:bayes-density-informal}. Suppose \(\pi\) is itself represented by a density function \(p_\pi\) with respect to the Lebesgue measure \(\d x\). Then
\[
c^\dag_\pi (A|y) = \int_{x:A} \, \frac{p(y|x) \, p_\pi(x)}{\int_{x':X} \, p(y|x') \, p_\pi(x') \, \d x'} \, \d x.
\]

\subsection{Optics}
\label{sec:org7668931}
\label{sec:optics}

In \secref{sec:comp-prob} we noted that, given a channel \(c : X \klto Y\), its Bayesian inversion is of the form \(\bdag{c} : \cat{C}(I, X) \to \cat{C}(Y, X)\), where \(\cat{C}(I, X)\) is a space of states on \(X\). This is not a map in \(\Kl(\Dst)\), for instance, because there is in general no space \(Z\) such that \(\Kl(\Dst)(Y, X) \cong \Dst Z\); and nor do we obtain a map in \(\Kl(\Dst)\) if we attempt to `uncurry' \(\bdag{c}\) into the form \(\Dst X \otimes Y \to \Dst X\) \footnote{Not only is \(\Kl(\Dst)\) not categorically closed, but \(\bdag{c}\) is not linear in the prior: the Bayesian inversion of \(c\) with respect to \(0.5\pi + 0.5\rho\) is not \(0.5 c^\dag_\pi + 0.5 c^\dag_\rho\); such linearity characterizes maps in \(\Kl(\Dst)\). Alternatively, \(c^\dag\) is not generally a morphism in \(\Cat{sfKrn}\), because there may be some prior \(\pi\) such that \((c \klcirc \pi)(y) = 0\), which would make the required almost-inverse undefined, so that \(c^\dag\) is not the sum of at most countably many finite kernels.}. So, unlike in the case of Cartesian lenses, our forwards and backwards morphisms do not live in the same category, yet somehow they still interact and behave similarly: we need \emph{mixed optics}.

\emph{Mixed} or \emph{profunctor optics} \citep{Roman2020Profunctor,Clarke2020Profunctor,Riley2018Categories} allow the forwards and backwards morphisms of bidirectional transformations such as lenses to live in arbitrary (possibly different) categories \(\cat{C}\) and \(\cat{D}\), with interaction mediated by an arbitrary third category \(\cat{M}\) of `residuals'. The objects of \(\cat{M}\) can be somehow tensored with the objects of \(\cat{C}\) and \(\cat{D}\), giving new \(\cat{C}\) and \(\cat{D}\) objects that behave like the original objects plus ``some other stuff''; through this tensoring, we say that \(\cat{M}\) \emph{acts} on \(\cat{C}\) and \(\cat{D}\), and \(\cat{C}\) and \(\cat{D}\) are \(\cat{M}\text{-actegories}\). For example, recall that the \(\mathsf{view}\) map of a Cartesian lens takes a structure and returns a part of it; the residual (the ``other stuff'') in this case is just the rest of the record, and \(\Set\) is acting on itself.

Henceforth, rather than work in the setting of locally small categories enriched in \(\Set\), we will work in the somewhat more general setting of enrichment in an arbitrary cocomplete Cartesian closed category \(\Cat{V}\). We write \(\Cat{V\mdash Cat}\) for the category of \(\Cat{V}\text{-enriched}\) categories, so that \(\Cat{V\mdash Cat}(\cat{C}, \cat{D})\) is the \(\Cat{V}\text{-category}\) of \(\Cat{V}\text{-functors}\) between \(\Cat{V}\text{-categories}\). Since \(\Cat{V}\) is assumed to be Cartesian, we write \(\times\) for the categorical product both in \(\Cat{V}\) and the induced product in \(\Cat{V\mdash Cat}\).

\begin{defn}[$\cat{M}\text{-actegory}$]
Suppose \(\cat{M}\) is a monoidal category with tensor \(\otimes\) and unit object \(I\). We say that \(\cat{C}\) is an \(\cat{M}\text{\textbf{-actegory}}\) when \(\cat{C}\) is equipped with a functor \(\odot : \cat{M} \to \Cat{V\mdash Cat}(\cat{C}, \cat{C})\) called the \textbf{action} along with natural unitor and associator isomorphisms \(\lambda^\odot_X : I \odot X \xto{\sim} X\) and \(a^\odot_{M,N,X} : (M \otimes N) \odot X \xto{\sim} M \odot (N \odot X)\) compatible with the monoidal structure of \((\cat{M}, \otimes, I)\).
\end{defn}

\begin{defn}[Mixed optics \citep{Clarke2020Profunctor}] \label{def:optics}
Suppose \((\cat{C}, \circL)\) and \((\cat{D}, \circR)\) are two \(\cat{M}\text{-actegories}\). Let \(X, Y : \cat{C}\) and \(A, B : \cat{D}\). An \textbf{optic} from \((X, A)\) to \((Y, B)\), written \((X, A) \lensto (Y, B)\), is an element of the following object in \(\Cat{V}\):
\begin{equation} \label{eq:optics}
\Cat{Optic}_{\circL, \circR}\Big( (X, A), (Y, B) \Big) = \int^{M \, : \, \cat{M}} \cat{C}(X, M \circL Y) \times \cat{D}(M \circR B, A)
\end{equation}
\end{defn}

The `integral' here is not an integral but a \emph{coend}: a kind of generalized sum or existential quantifier; see
\textcite[Example 5.4]{Loregian2015This} or \textcite[Chapter 4]{Fong2018Seven}
for some background to this intuition. The coend ranges over objects \(M : \cat{M}\), binding pairs of morphisms \(X \to M \circL Y\) in \(\cat{C}\) and \(M \circR B \to A\) in \(\cat{D}\) into equivalence classes along the residuals \(M\). Let \(v : \cat{C}(X, M \circL Y)\), \(u : \cat{D}(N \circR B, A)\). Then, for any \(f : \cat{M}(M, N)\), we have two pairs of morphisms
\[
\optar{v}{u \circ (f \circR \id_{B})} \; := \;
\big( v, \; u \circ (f \circR \id_{B}) \big)
\; : \; \cat{C}(X, M \circL Y) \times \cat{D}(M \circR B, A)
\]
and
\[
\optar{(f \circL \id_Y) \circ v}{u} \; := \;
\big( (f \circL \id_Y) \circ v, \; u \big)
\; : \; \cat{C}(X, N \circL Y) \times \cat{D}(N \circR B, A) \, .
\]
We give a recap of the definition of coend in \secref{sec:coends}. In brief, the coend equivalence relation says precisely that two such pairs are equivalent, and so we call \(\optar{v}{u \circ (f \circR \id_{B})}\) and \(\optar{(f \circL \id_Y) \circ v}{u}\) \emph{representatives} of their equivalence class. We adopt the notation \(\optar{l}{r}\) to indicate the element of the coend (\emph{i.e.}, the equivalence class) represented by the pair \((l, r)\).

Apart from providing a unified compositional framework for describing bidirectional transformations, optics admit an intuitive graphical calculus \citep{Boisseau2020String,Roman2020Open}. A general optic \(\optar{l}{r} : (X, A) \lensto (Y, B)\) is depicted\footnote{For these diagrams we adopt the graphical calculus of \citet{Boisseau2020String} of the bicategory of Tambara modules, which are presheaves of optics. 0-cells are actegories, depicted as planar regions. 1-cells are Tambara modules, depicted as edges of regions (\emph{i.e.}, strings). 2-cells are natural transformations, depicted as vertices on edges (\emph{i.e.}, boxes on strings). For our purposes, these 2-cells will always be morphisms in an underlying actegory, lifted by the Yoneda embedding. The graphical calculus described by \citet{Roman2020Open} is more flexible, representing the monoidal bicategory of pointed profunctors without the extra Tambara module structure, but here we follow \citet{Boisseau2020String} for simplicity.} as
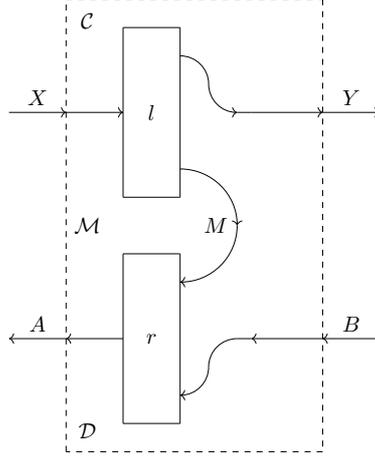
\[
\scalebox{0.75}{\input{img/lens-general-coend-0.tikz}}
\]
where the top region of the diagram represents \(\cat{C}\), the middle region \(\cat{M}\), and the bottom region \(\cat{D}\). Information flows from left to right in the top region, and right to left in the bottom, and \(\cat{M}\) mediates interaction between \(\cat{C}\) and \(\cat{D}\). We can depict the equivalent representatives \(\optar{v}{u \circ (f \circR \id_{B})} \sim \optar{(f \circL \id_Y) \circ v}{u}\) accordingly as
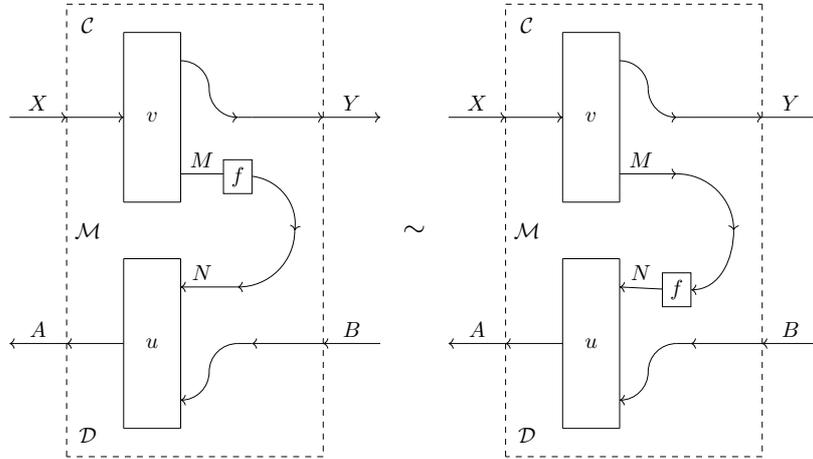
\[
\scalebox{0.75}{\input{img/lens-general-coend-1.tikz}}
\; \sim \;
\scalebox{0.75}{\input{img/lens-general-coend-2.tikz}}
\]
which indicates that two pairs of morphisms are equivalent under the coend when there is some \(f\) that can `slide between' residual types.

As these diagrams suggest, optics for \(\circL\) and \(\circR\) form a category: composition is by pasting of diagrams, and identities are plain wires.

\begin{prop}[Category of optics {\parencite[3.1.1]{Roman2020Profunctor}}]
Given \(\cat{M}\text{-actegories, } (\cat{C}, \circL)\) and \((\cat{D}, \circR)\), there is a \textbf{category of optics} \(\Cat{Optic}_{\circL, \circR}\) whose objects are pairs of objects \((X, A) : (\cat{C} \times \cat{D})_0\) and whose morphisms \((X, A) \lensto (Y, B)\) are elements of \(\Cat{Optic}_{\circL, \circR}\Big( (X, A), (Y, B) \Big)\) as defined in \eqref{eq:optics}. The (representative of) the composition of two optics is as depicted in the following diagram. Let \(\optar{v}{u} : (X, A) \lensto (Y, B)\) and \(\optar{l}{r} : (Y, B) \lensto (Z, C)\). Then \(\optar{l}{r} \lenscirc \optar{v}{u} : (X, A) \lensto (Z, C) \; \cong\)
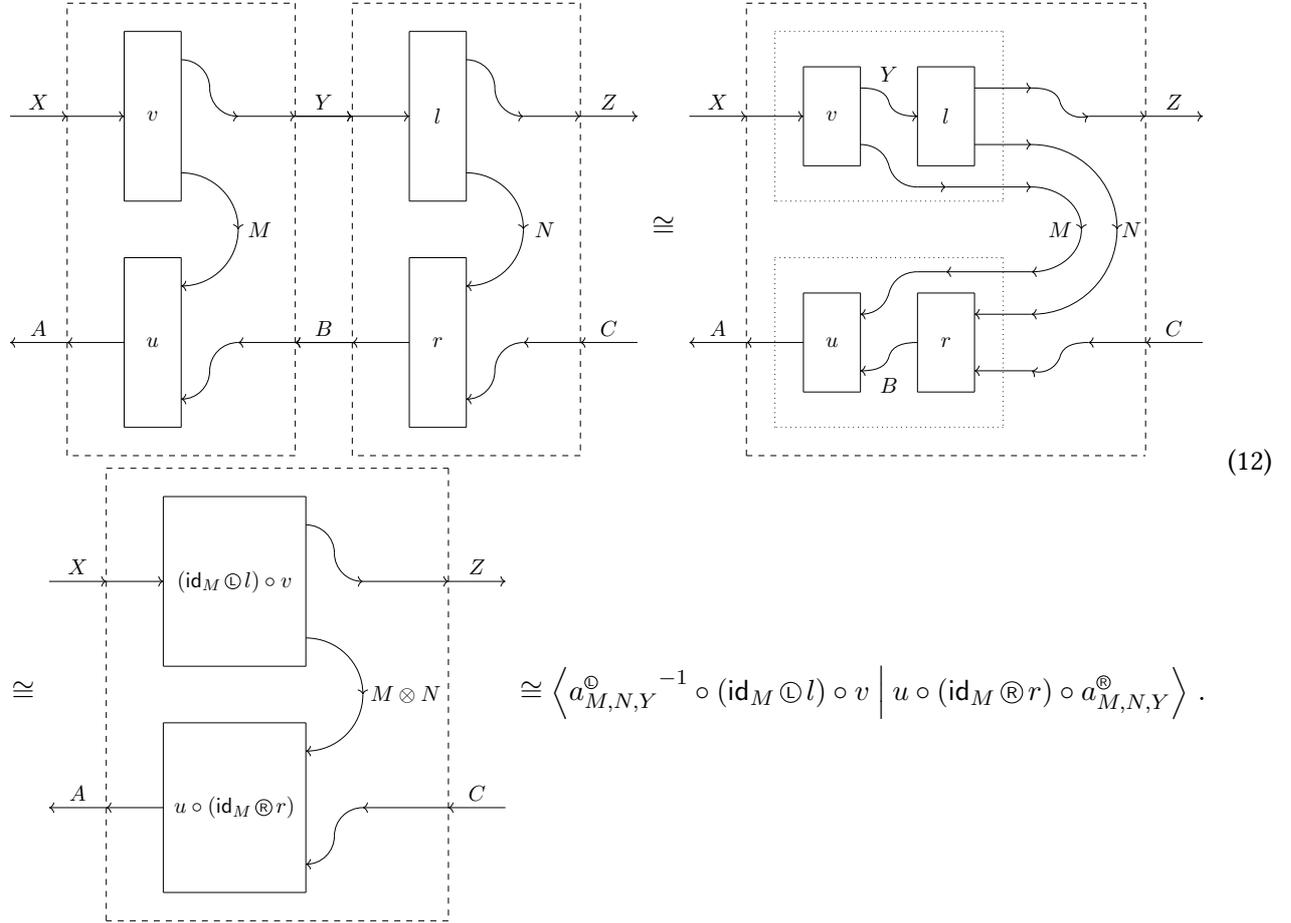
\begin{equation} \label{eq:optic-composition}
\begin{aligned}
&\scalebox{0.75}{\input{img/lens-general-lrvu-1.tikz}}
 \cong \scalebox{0.75}{\input{img/lens-general-lrvu-2.tikz}} \\
&\cong \scalebox{0.75}{\input{img/lens-general-lrvu-0.tikz}}
 \cong \optar{{a^{\circL}_{M,N,Y}}^{-1} \circ \left( \id_M \circL \, l \right) \circ v}{u \circ (\id_M \circR \, r) \circ a^{\circR}_{M,N,Y}} \, .
\end{aligned}
\end{equation}
Identity optics \(\id_{(X, A)} : (X, A) \lensto (X, A)\) are given by the unitors of the actegory structures: \(\id_{(X, A)} = \optar{{\lambda^{\circL}_X}^{-1}}{\lambda^{\circR}_A}\), depicted as plain wires in an otherwise empty box:
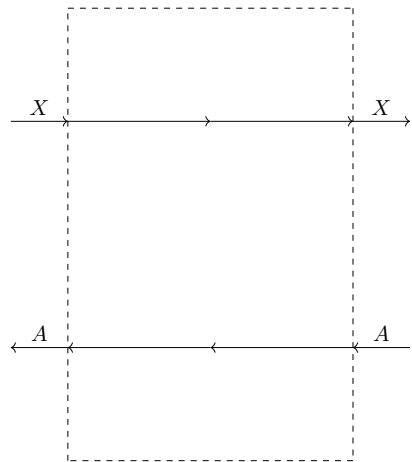
\[
\scalebox{0.75}{\input{img/lens-cartesian-id-2.tikz}}
\]
\end{prop}

\subsubsection{Lenses}
\label{sec:org4f98ee1}

A Cartesian lens as introduced in \secref{sec:intro} is a pair of functions \(X \to Y\) and \(X \times B \to A\); that is, an element of the product \(\Set(X, Y) \times \Set(X \times B, A)\). We can write this in optical form:
\begin{align}
\Set(X, Y) \times \Set(X \times B, A)
& \cong \int^{M \, : \, \Set} \Set(X, Y) \times \Set(X, M) \times \Set(M \times B, A) \label{eq:cartlens1} \\
& \cong \int^{M \, : \, \Set} \Set(X, M \times Y) \times \Set(M \times B, A) \label{eq:cartlens2} \\
& \cong \Cat{Optic}_{\times, \times}\Big( (X, A), (Y, B) \Big) \nonumber
\end{align}
where the first isomorphism obtains by Yoneda reduction \eqref{eq:coyoneda} and the second by the universal property of the categorical product \(\times : \Set \to \Cat{Cat}(\Set, \Set)\).

The universal property of the Cartesian product that justifies \eqref{eq:cartlens2} \(\xto{\sim}\) \eqref{eq:cartlens1} entails that \(\Set\) supplies comonoids and every morphism in \(\Set\) is a comonoid homomorphism: \emph{i.e.}, \(\copier \circ f = f \otimes f \circ \copier\), where \(\copier : x \mapsto (x, x)\) is the diagonal copier in \(\Set\). When either of the \(\cat{M}\text{-actegories}\) underlying a category of optics is equivalent to \(\cat{M}\) itself, we can lift string diagrams in that actegory directly into the string diagrams for those optics \parencite[Note 3.7]{Boisseau2020String}. In particular, this includes the depictions of comonoids introduced in \secref{sec:graph-calc}. We can thus depict any Cartesian lens as
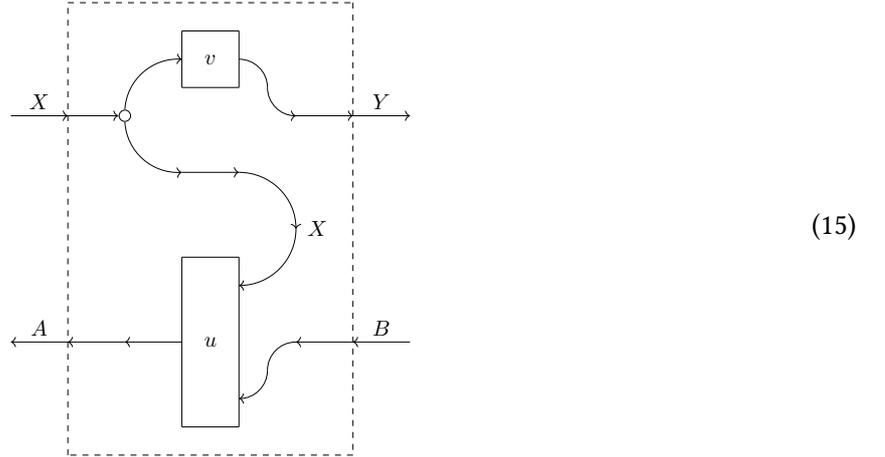
\begin{equation} \label{eq:lens-optic}
\scalebox{0.75}{\input{img/lens-cartesian-vu.tikz}}
\end{equation}
where \(v\) is called \(\mathsf{view}\) and \(u\) is called \(\mathsf{update}\). We can define a general lens to be any optic that is isomorphic to such a depiction.
\begin{defn}[After {\textcite[{\S}3.1]{Clarke2020Profunctor}}] \label{def:lens}
A \textbf{lens} is any optic that can be depicted as in \eqref{eq:lens-optic}. Equivalently, suppose \((\cat{C}, \otimes)\) is a symmetric monoidal category and write \(\Cat{Comon}(\cat{C})\) for its subcategory of comonoids and comonoid homomorphisms. \(\otimes\) lifts to \(\Cat{Comon}(\cat{C})\) and induces a corresponding \(\Cat{Comon}(\cat{C})\text{-actegory}\) structure on \(\Cat{Comon}(\cat{C})\). Suppose also that \((\cat{D}, \circR)\) is any \(\Cat{Comon}(\cat{C})\text{-actegory}\). Then a lens is any optic in \(\Cat{Optic}_{(\otimes, \circR)}\). Note that
\begin{align*}
\Cat{Optic}_{\otimes, \circR}\Big( (X, A), (Y, B) \Big)
& \cong \int^{M \, : \, \Cat{Comon}(\cat{C})} \Cat{Comon}(\cat{C})(X, M \otimes Y) \times \cat{D}(M \circR B, A) \\
& \cong \int^{M \, : \, \Cat{Comon}(\cat{C})} \Cat{Comon}(\cat{C})(X, Y) \times \Cat{Comon}(\cat{C})(X, M) \times \cat{D}(M \circR B, A) \\
& \cong \Cat{Comon}(\cat{C})(X, Y) \times \cat{D}(X \circR B, A)
\end{align*}
where the second isomorphism follows because \(\copier \circ f \cong f \otimes f \circ \copier\) for every morphism \(f\) in \(\Cat{Comon}(\cat{C})\) and the third follows by Yoneda reduction \eqref{eq:coyoneda}. Every such optic therefore has a representative as depicted in \eqref{eq:lens-optic}. \qed
\end{defn}

In the sequel, we will see that Bayesian inversions constitute the `backwards' components of a particular category of lenses.

\section{Channels relative to a state}
\label{sec:orgf1ace06}
\label{sec:stat-cat}

The Bayesian inversion of a `forward' channel is defined with respect to a prior state on the domain of the forward channel. Changes in the prior entail changes in the inversions -- but ``changes in the prior'' are just channels in the forwards direction, and the ``changes in the inversions'' correspond to pulling inversions back along corresponding forward channels. Formally, this means that the backward channels are fibred over the forward channels: for each domain in the `base category' of forward channels, we have a category of channels with respect to that domain, and forward channels correspond to contravariant functors between the fibres that implement the aforesaid pulling-back. This is an instance of the Grothendieck construction \citep{nLab2020Grothendieck}, making Bayesian lenses an instance of \emph{Grothendieck lenses} \citep{Spivak2019Generalized}. In this section, we make these ideas precise; in the next, we translate them into the optical vernacular introduced in \secref{sec:optics}.

\begin{defn}[State-indexed categories] \label{def:stat-cat}
Let \((\cat{C}, \otimes, I)\) be a monoidal category enriched in a Cartesian closed category \(\Cat{V}\). Define the \(\cat{C}\text{-state-indexed}\) category \(\Fun{Stat}: \cat{C}\op \to \Cat{V\mdash Cat}\) as follows. 
\begin{align}
\Fun{Stat} \;\; : \;\; \cat{C}\op \; & \to \; \Cat{V\mdash Cat} \nonumber \\
X & \mapsto \Fun{Stat}(X) := \quad \begin{pmatrix*}[l]
& \Fun{Stat}(X)_0 & := \quad \;\;\; \cat{C}_0 \\
& \Fun{Stat}(X)(A, B) & := \quad \;\;\; \Cat{V}(\cat{C}(I, X), \cat{C}(A, B)) \\
\id_A \: : & \Fun{Stat}(x)(A, A) & := \quad 
\left\{ \begin{aligned}
\id_A : & \; \cat{C}(I, X)     \to     \cat{C}(A, A) \\
        & \quad\;\;\: \rho \quad \mapsto \quad \id_A
\end{aligned} \right. \label{eq:stat} \\
\end{pmatrix*} \\ \nonumber \\
f : \cat{C}(Y, X) & \mapsto \begin{pmatrix*}[c]
\Fun{Stat}(f) \; : & \Fun{Stat}(X) & \to & \Fun{Stat}(Y) \vspace*{0.5em} \\
& \Fun{Stat}(X)_0 & = & \Fun{Stat}(Y)_0 \vspace*{0.5em} \\
& \Cat{V}(\cat{C}(I, X), \cat{C}(A, B)) & \to & \Cat{V}(\cat{C}(I, Y), \cat{C}(A, B)) \vspace*{0.125em} \\
& \alpha & \mapsto & f^\ast \alpha : \big( \, \sigma : \cat{C}(I, Y) \, \big) \mapsto \big( \, \alpha(f \klcirc \sigma) : \cat{C}(A, B) \, \big)
\end{pmatrix*} \nonumber
\end{align}
Composition in each fibre \(\Fun{Stat}(X)\) is given by composition in \(\cat{C}\); that is, by the left and right actions of the profunctor \(\Fun{Stat}(X)(-, =) : \cat{C}\op \times \cat{C} \to \Cat{V}\) (\secref{sec:coends} supplies some intuition). Explicitly, given \(\alpha : \Cat{V}(\cat{C}(I, X), \cat{C}(A, B))\) and \(\beta : \Cat{V}(\cat{C}(I, X), \cat{C}(B, C))\), their composite is \(\beta \circ \alpha : \Cat{V}(\cat{C}(I, X), \cat{C}(A, C)) : = \rho \mapsto \beta(\rho) \klcirc \alpha(\rho)\). Since \(\Cat{V}\) is Cartesian, there is a canonical copier \(\copier : x \mapsto (x, x)\) on each object, so we can alternatively write \((\beta \circ \alpha)(\rho) = \big(\beta(-) \klcirc \alpha(-)\big) \circ \copier \circ \rho\). Note that we indicate composition in \(\cat{C}\) by \(\klcirc\) and composition in the fibres \(\Fun{Stat}(X)\) by \(\circ\).
\end{defn}

\begin{ex} \label{ex:stat-meas}
Let \(\Cat{V} = \Cat{Meas}\) be a `convenient' (\emph{i.e.}, Cartesian closed) category of measurable spaces, such as the category of quasi-Borel spaces \citep{Heunen2017Convenient}, let \(\Pow : \Cat{Meas} \to \Cat{Meas}\) be a probability monad defined on this category, and let \(\cat{C} = \Kl(\Pow)\) be the Kleisli category of this monad. Its objects are the objects of \(\Cat{Meas}\), and its hom-spaces \(\Kl(\Pow)(A, B)\) are the spaces \(\Cat{Meas}(A, \Pow B)\) \citep{Fritz2019synthetic}. This \(\cat{C}\) is a monoidal category of stochastic channels, whose monoidal unit \(I\) is the space with a single point. Consequently, states of \(X\) are just measures (distributions) in \(\Pow X\). That is, \(\Kl(\Pow)(I, X) \cong \Cat{Meas}(1, \Pow X)\). Instantiating  \(\Fun{Stat}\) in this setting, we obtain:
\begin{align}
\Fun{Stat} \;\; : \;\; \Kl(\Pow)\op \; & \to \; \Cat{V\mdash Cat} \nonumber \\
X & \mapsto \Fun{Stat}(X) := \quad \begin{pmatrix*}[l]
& \Fun{Stat}(X)_0 & := \quad \;\;\; \Cat{Meas}_0 \\
& \Fun{Stat}(X)(A, B) & := \quad \;\;\; \Cat{Meas}(\Pow X, \Cat{Meas}(A, \Pow B)) \\
\id_A \: : & \Fun{Stat}(X)(A, A) & := \quad
\left\{ \begin{aligned}
\id_A : & \; \Pow X     \to     \Cat{Meas}(A, \Pow A) \\
        & \;\;\; \rho \;\;\, \mapsto \quad \eta_A
\end{aligned} \right. \label{eq:stat-kl-d} \\
\end{pmatrix*} \\ \nonumber \\
c : \Kl(\Pow)(Y, X) & \mapsto \begin{pmatrix*}[c]
\Fun{Stat}(c) \; : & \Fun{Stat}(X) & \to & \Fun{Stat}(Y) \vspace*{0.5em} \\
& \Fun{Stat}(X)_0 & = & \Fun{Stat}(Y)_0 \vspace*{0.5em} \\
& \begin{pmatrix*}[l]
    d^\dag : & \Pow X & \to \Kl(\Pow)(A, B) \\
    & \; \pi & \mapsto \quad \quad d^\dag_\pi
  \end{pmatrix*}
  & \mapsto &
  \begin{pmatrix*}
    c^\ast d^\dag : \Pow Y \to \Kl(\Pow)(A, B) \\
    \rho \quad \mapsto \quad d^\dag_{c \klcirc \rho}
  \end{pmatrix*}
\end{pmatrix*} \nonumber
\end{align}
Each \(\Fun{Stat}(X)\) is a category of stochastic channels with respect to measures on the space \(X\). We can write morphisms \(d^\dag : \Pow X \to \Kl(\Pow)(A, B)\) in \(\Fun{Stat}(X)\) as \(d^\dag_{(\cdot)} : A \xklto{(\cdot)} B\), and think of them  as generalized Bayesian inversions: given a measure \(\pi\) on \(X\), we obtain a channel \(d^\dag_\pi : A \xklto{\pi} B\) with respect to \(\pi\). Given a channel \(c : Y \klto X\) in the base category of priors, we can pull \(d^\dag\) back along \(c\), to obtain a \(Y\text{-dependent}\) channel in \(\Fun{Stat}(Y)\), \(c^\ast d^\dag : \Pow Y \to \Kl(\Pow)(A, B)\), which takes \(\rho : \Pow Y\) to the channel \(d^\dag_{c \klcirc \rho} : A \xklto{c \klcirc \rho} B\) defined by pushing \(\rho\) through \(c\) and then applying \(d^\dag\).
\end{ex}

\begin{rmk}
Note that by taking \(\Cat{Meas}\) to be Cartesian closed, we have \(\Cat{Meas}(\Pow X, \Cat{Meas}(A, \Pow B)) \cong \Cat{Meas}(\Pow X \times A, \Pow B)\) for each \(X\), \(A\) and \(B\), and so a morphism \(c^\dag : \Pow Y \to \Kl(\Pow)(X, Y)\) equivalently has the type \(\Pow Y \times X \to \Pow Y\). Paired with a channel \(c : Y \to \Pow X\), we have something like a Cartesian lens; and to compose such pairs, we can use the Grothendieck construction \citep{nLab2020Grothendieck,Spivak2019Generalized}.

\end{rmk}

\begin{defn}[Grothendieck lenses \citep{Spivak2019Generalized}]
We define the category \(\Cat{GrLens}_F\) of Grothendieck lenses for a (pseudo)functor \(F : \cat{C}\op \to \Cat{V\mdash Cat}\) to be the total category of the Grothendieck construction for the pointwise opposite of \(F\). Explicitly, its objects \((\Cat{GrLens}_F)_0\) are pairs \((C, X)\) of objects \(C\) in \(\cat{C}\) and \(X\) in \(F(C)\), and its hom-sets \(\Cat{GrLens}_F \big( (C, X), (C', X') \big)\) are given by dependent sums
\begin{equation}
\Cat{GrLens}_F \big( (C, X), (C', X') \big) = \sum_{f \, : \, \cat{C}(C, C')} F(C) \big( F(f)(X'), X \big)
\end{equation}
so that a morphism \((C, X) \lensto (C', X')\) is a pair \((f, f^\dag)\) of \(f : \cat{C}(C, C')\) and \(f^\dag : F(C) \big( F(f)(X'), X \big)\). We call such pairs \textbf{Grothendieck lenses} for \(F\) or \(F\mathrm{\textit{-lenses}}\).
\end{defn}

\begin{prop}[$\Cat{GrLens}_F$ is a category]
The identity Grothendieck lens on \((C, X)\) is \(\id_{(C, X)} = (\id_C, \id_X)\). Sequential composition is as follows. Given \((f, f^\dag) : (C, X) \lensto (C', X')\) and \((g, g^\dag) : (C', X') \lensto (D, Y)\), their composite \((g, g^\dag) \lenscirc (f, f^\dag)\) is defined to be the lens \(\big(g \klcirc f, F(f)(g^\dag) \big) : (C, X) \lensto (D, Y)\). Associativity and unitality of composition follow from functoriality of \(F\). \qed
\end{prop}

\begin{ex}[$\Cat{GrLens}_{\Fun{Stat}}$] \label{ex:stat-lens}
Instantiating \(\Cat{GrLens}_F\) with \(F = \Fun{Stat} : \cat{C}\op \to \Cat{V\mdash Cat}\), we obtain the category \(\Cat{GrLens}_\Fun{Stat}\) whose objects are pairs \((X, A)\) of objects of \(\cat{C}\) and whose morphisms \((X, A) \lensto (Y, B)\) are elements of the set
\begin{equation}
\Cat{GrLens}_\Fun{Stat} \big( (X, A), (Y, B) \big) \cong \cat{C}(X, Y) \times \Cat{V} \big( \cat{C}(I, X), \cat{C}(B, A) \big) \, .
\end{equation}
The identity \(\Fun{Stat}\text{-lens}\) on \((Y, A)\) is \((\id_Y, \id_A)\), where by abuse of notation \(\id_A : \cat{C}(I, Y) \to \cat{C}(A, A)\) is the constant map \(\id_A\) defined in \eqref{eq:stat} that takes any state on \(Y\) to the identity on \(A\). The sequential composite of \((c, c^\dag) : (X, A) \lensto (Y, B)\) and \((d, d^\dag) : (Y, B) \lensto (Z, C)\) is the \(\Fun{Stat}\text{-lens } \big( (d \klcirc c), (c^\dag \circ c^\ast d^\dag) \big) : (X, A) \lensto (Z, C)\) with \((d \klcirc c) : \cat{C}(X, Z)\) and where \((c^\dag \circ c^\ast d^\dag) : \Cat{V}\big(\cat{C}(I, X), \cat{C}(C, A)\big)\) takes a state \(\pi : \cat{C}(I, X)\) on \(X\) to the channel \(c^\dag_{\pi} \klcirc \d^\dag_{c \klcirc \pi}\). If we think of the notation \((\cdot)^\dag\) as denoting the operation of forming the Bayesian inverse of a channel (in the case where \(A = X\), \(B = Y\) and \(C = Z\)), then the main result of this paper is to show that \((d \klcirc c)^\dag_\pi \overset{d \klcirc c \klcirc \pi}{\sim} c^\dag_{\pi} \klcirc \d^\dag_{c \klcirc \pi}\), where \(\overset{d \klcirc c \klcirc \pi}{\sim}\) denotes \((d \klcirc c \klcirc \pi)\text{-almost-equality}\) (Definition \ref{def:almost-eq}).
\end{ex}

\section{Bayesian lenses}
\label{sec:org0e1c1c7}
\label{sec:bayes-lens}

We now show how to translate the categories of Grothendieck \(\Fun{Stat}\text{-lenses}\) defined above into the canonical profunctor optic form, thereby opening Bayesian lenses up to comparison and composition with other optics, and representation in the corresponding graphical calculi.

In order to give an optical form for \(\Cat{GrLens}_\Fun{Stat}\), we need to find two \(\cat{M}\text{-actegories}\) with a common category of actions \(\cat{M}\). Let \(\hat{\cat{C}}\) and \(\check{\cat{C}}\) denote the categories \(\hat{\cat{C}} := \Cat{V\mdash Cat}(\cat{C}\op, \Cat{V})\) and \(\check{\cat{C}} := \Cat{V\mdash Cat}(\cat{C}, \Cat{V})\) of presheaves and copresheaves on \(\cat{C}\), and consider the following natural isomorphisms.
\begin{align}
\Cat{GrLens}_\Fun{Stat} \big( (X, A), (Y, B) \big) & \cong \cat{C}(X, Y) \times \Cat{V} \big( \cat{C}(I, X), \cat{C}(B, A) \big) \nonumber \\
& \cong \int^{M \, : \, \cat{C}} \cat{C}(X, Y) \times \cat{C}(X, M) \times \Cat{V}\big(\cat{C}(I, M), \cat{C}(B, A)\big) \nonumber \\
& \cong \int^{\hat{M} \, : \, \hat{\cat{C}}} \cat{C}(X, Y) \times \hat{M}(X) \times \Cat{V}\big(\hat{M}(I), \cat{C}(B, A)\big) \label{eq:stat-lens-coend}
\end{align}
The second isomorphism follows by Yoneda reduction \eqref{eq:coyoneda}, and the third follows by the Yoneda lemma. We take \(\cat{M}\) to be \(\cat{M} := \hat{\cat{C}}\), and define an action \(\odot\) of \(\hat{\cat{C}}\) on \(\check{\cat{C}}\) as follows.
\begin{defn}[$\odot$]
We give only the action on objects; the action on morphisms is analogous.
\begin{equation} \label{eq:L-action}
\begin{aligned}
\odot : \hat{\cat{C}} & \to \Cat{V\mdash Cat}(\check{\cat{C}}, \check{\cat{C}}) \\
\hat{M} & \mapsto
  \begin{pmatrix*}
    \hat{M} \odot - & : & \check{\cat{C}} & \to & \check{\cat{C}} \\
    & & P & \mapsto & \Cat{V}\big( \hat{M}(I), P \big)
  \end{pmatrix*}
\end{aligned}
\end{equation}
Functoriality of \(\odot\) follows from the functoriality of copresheaves. \qed
\end{defn}

To confirm that \(\odot\) makes \(\check{\cat{C}}\) into a \(\hat{\cat{C}}\text{-actegory}\), we need to check the actegory structure isomorphisms.

\begin{prop} \label{prop:odot-actegory}
\(\odot\) equips \(\check{\cat{C}}\) with a \(\hat{\cat{C}}\text{-actegory}\) structure: unitor isomorphisms \(\lambda^{\odot}_F : 1 \odot F \xto{\sim} F\) and associator isomorphisms \(a^{\odot}_{\hat{M}, \hat{N}, F} : (\hat{M} \times \hat{N}) \odot F \xrightarrow{\sim} \hat{M} \odot (\hat{N} \odot F)\) for each \(\hat{M},\hat{N}\) in \(\check{\cat{C}}\), both natural in \(F : \Cat{V\mdash Cat}(\cat{C}, \Cat{V})\).
\begin{proof}
We first check the unitor:
\begin{align*}
\lambda^{\odot}_F : 1 \odot \cat{C}(B, -) 
& = \Cat{V}\big(1(I), \cat{C}(B, -)\big) \\
& \cong \Cat{V}\big(\mathbf{1}, \cat{C}(B, -)\big) \\
& \cong \cat{C}(B, -)
\end{align*}
where $\mathbf{1}$ is the terminal object in $\Cat{V}$.

The associator is given as follows:
\begin{align*}
{a^{\odot}_{\hat{M}, \hat{N}, P}}^{-1}
: \hat{M} \odot \left( \hat{N} \odot P \right)
& = \Cat{V}\left(\hat{M}(I), \Cat{V}\left(\hat{N}(I), P\right)\right) \\
& \cong \Cat{V}\left(\hat{M}(I) \times \hat{N}(I), P\right) \\
& \cong \Cat{V}\left((\hat{M}\times\hat{N})(I), P\right) \\
& = (\hat{M} \times \hat{N}) \odot P
\end{align*}
where the first isomorphism follows by the Cartesian closure of \(\Cat{V}\).
\end{proof}
\end{prop}

We are now in a position to define the category of abstract Bayesian lenses, and show that this category coincides with the category of \(\Fun{Stat}\text{-lenses}\).
\begin{defn}[Bayesian lenses] \label{def:bayes-lens}
Denote by \(\Cat{BayesLens}\) the category of optics \(\Cat{Optic}_{\times, \odot}\) for the action of the Cartesian product on presheaf categories \(\times : \hat{\cat{C}} \to \Cat{V\mdash Cat}(\hat{\cat{C}}, \hat{\cat{C}})\) and the action \(\odot : \hat{\cat{C}} \to \Cat{V\mdash Cat}(\check{\cat{C}}, \check{\cat{C}})\) defined in \eqref{eq:L-action}. Its objects \((\hat{X}, \check{Y})\) are pairs of a presheaf and a copresheaf on \(\cat{C}\), and its morphisms \((\hat{X}, \check{A}) \lensto (\hat{Y}, \check{B})\) are abstract \textbf{Bayesian lenses}---elements of the set
\begin{equation*}
\Cat{Optic}_{\times, \odot}\Big((\hat{X}, \check{A}), (\hat{Y}, \check{B})\Big)
= \int^{\hat{M} \, : \, \hat{\cat{C}}} \hat{\cat{C}}(\hat{X}, \hat{M} \times \hat{Y}) \times \check{\cat{C}}(\hat{M} \odot \check{B}, \check{A})
\end{equation*}
A Bayesian lens \((\hat{X}, \check{X}) \lensto (\hat{Y}, \check{Y})\) is called a \textbf{simple} Bayesian lens.
\end{defn}

\begin{prop}
\(\Cat{BayesLens}\) is a category of lenses.
\begin{proof}
The product \(\times : \hat{\cat{C}} \to \Cat{V\mdash Cat}(\hat{\cat{C}}, \hat{\cat{C}})\) on \(\hat{\cat{C}}\) is Cartesian, so \(\Cat{Comon}(\hat{\cat{C}}) = \hat{\cat{C}}\). Hence
\begin{align}
\Cat{Optic}_{\times, \odot}\Big((\hat{X}, \check{A}), (\hat{Y}, \check{B})\Big)
& \cong \int^{\hat{M} \, : \, \hat{\cat{C}}} \hat{\cat{C}}(\hat{X}, \hat{Y}) \times \hat{\cat{C}}(\hat{X}, \hat{M}) \times \check{\cat{C}}(\hat{M} \odot \check{B}, \check{A}) \label{eq:cart-bayes}
\end{align}
is of the form in definition \ref{def:lens}.
\end{proof}
\end{prop}

\begin{prop}[$\Fun{Stat}\text{-lenses}$ are Bayesian lenses]
Let \(\hat{(\cdot)} : \cat{C} \hookrightarrow \Cat{V\mdash Cat}(\cat{C}\op, \Cat{V})\) denote the Yoneda embedding and \(\check{(\cdot)} : \cat{C} \hookrightarrow \Cat{V\mdash Cat}(\cat{C}, \Cat{V})\) the coYoneda embedding. Then
\begin{equation*}
\Cat{Optic}_{\times, \odot}\Big((\hat{X}, \check{A}), (\hat{Y}, \check{B})\Big)
\cong
\Cat{GrLens}_\Fun{Stat} \Big( (X, A), (Y, B) \Big)
\end{equation*}
so that \(\Cat{GrLens}_\Fun{Stat}\) is equivalent to the full subcategory of \(\Cat{Optic}_{\times, \odot}\) on representable (co)presheaves.
\begin{proof}
\begin{align*}
\Cat{Optic}_{\times, \odot}\Big((\hat{X}, \check{A}), (\hat{Y}, \check{B})\Big)
& \cong \int^{\hat{M} \, : \, \hat{\cat{C}}} \hat{\cat{C}}(\hat{X}, \hat{Y}) \times \hat{\cat{C}}(\hat{X}, \hat{M}) \times \check{\cat{C}}(\hat{M} \odot \check{B}, \check{A}) \\
& \cong \int^{\hat{M} \, : \, \hat{\cat{C}}} \hat{\cat{C}}(\hat{X}, \hat{Y}) \times \hat{\cat{C}}(\hat{X}, \hat{M}) \times \check{\cat{C}}\left(\Cat{V}(\hat{M}(I), \check{B}), \check{A}\right) \\
& \cong \int^{\hat{M} \, : \, \hat{\cat{C}}} \cat{C}(X, Y) \times \hat{M}(X) \times \Cat{V}\left(\hat{M}(I), \cat{C}(B, A)\right) \\
& \cong \Cat{GrLens}_\Fun{Stat} \Big( (X, A), (Y, B) \Big)
\end{align*}
The first isomorphism is just \eqref{eq:cart-bayes}, the second obtains by definition of \(\odot\), the third by the Yoneda lemma, and the fourth by \eqref{eq:stat-lens-coend}.

Since Bayesian lenses are lenses, we can check diagrammatically that sequential composition in \(\Cat{Optic}_{\times, \odot}\) corresponds to that in \(\Cat{GrLens}_\Fun{Stat}\). The composite lens \(\optar{d}{d^\dag} \lenscirc \optar{c}{c^\dag}\) of \(\optar{d}{d^\dag} : (Y, B) \lensto (Z, C)\) after \(\optar{c}{c^\dag} : (X, A) \lensto (Y, B)\) has the depiction
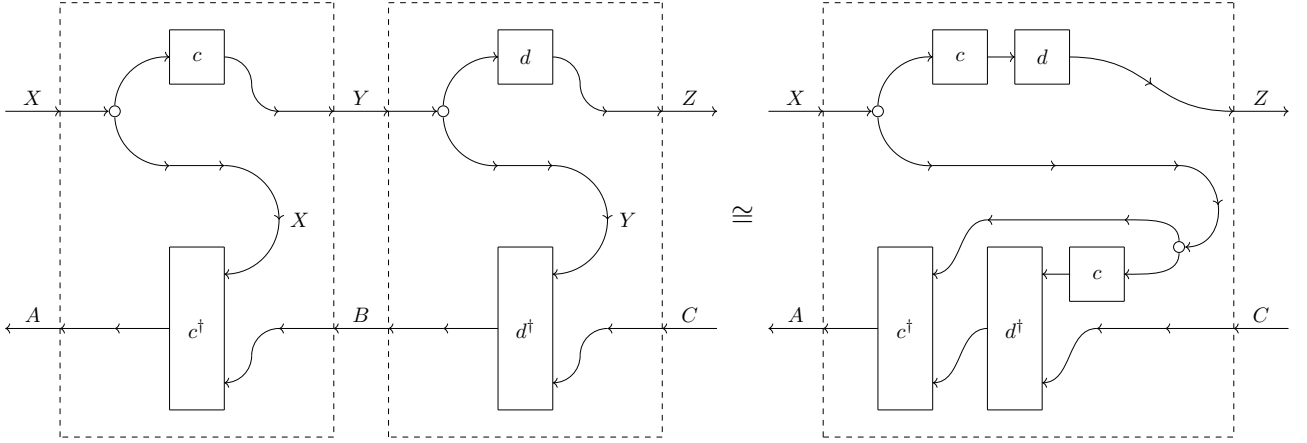
\begin{equation*}
\scalebox{0.72}{\input{img/lens-cartesian-dc-1.tikz}}
\cong
\scalebox{0.72}{\input{img/lens-cartesian-dc-2b.tikz}}
\end{equation*}
where the copier \(\copier\) is the universal map with diagonal components \(x \mapsto (x, x)\) induced by the Cartesian product \(\times\) on \(\hat{C}\); recall from the discussion preceding \eqref{eq:lens-optic} that we can lift diagrams in \((\hat{C}, \times)\) to diagrams in \(\Cat{Optic}_{\times, \odot}\). The isomorphism therefore follows because every morphism in \(\hat{C}\) is canonically a comonoid homomorphism, and we can slide morphisms along the optical residual.

From the right-hand side, we can read that the \(\mathsf{view}\) component of the composite optic is represented by \(d \klcirc c\) and the \(\mathsf{update}\) component is represented by
\begin{align*}
& c^\dag \circ \big(\id_{\check{X}} \odot \, \d^\dag \, \big) \circ a^\odot_{\hat{X}, \hat{Y}, \check{X}} \circ \big(\id_{\hat{X}} \times \, c \, \big) \circ \copier \\ & \; \cong \;
\big(c^\dag_{(-)} \klcirc \d^\dag_{c \klcirc (-)}\big) \circ \copier  \\ & \; \cong \;
c^\dag \circ c^\ast d^\dag
\end{align*}
where the first expression is given by reading the right-hand side following the definition of optical composition \eqref{eq:optic-composition}; where the first isomorphism follows by the definitions of \(\odot\), \(a^\odot_{\hat{X}, \hat{Y}, \check{X}}\), and the notation \(c^\dag_{(-)}\) formally defined in Example \ref{ex:stat-meas}; and where the second isomorphism follows by the definition of \(c^\dag_{(-)}\) and the definition of fibrewise composition in Definition \ref{def:stat-cat}.

We therefore have \(\optar{d}{d^\dag} \lenscirc \optar{c}{c^\dag} \cong \optar{d \klcirc c}{c^\dag \circ c^\ast d^\dag}\), which are just the components of the corresponding composite \(\Fun{Stat}\text{-lens}\) (Example \ref{ex:stat-lens}), and so \(\Fun{Stat}\text{-lenses}\) are Bayesian lenses.
\end{proof}
\end{prop}

\begin{rmk}
We will often abuse notation by indicating representable objects in \(\Cat{BayesLens}\) by their representations in \(\cat{C}\). That is, we will write \((X, A)\) instead of \((\hat{X}, \check{A})\) where this would be unambiguous.
\end{rmk}

It may be of interest sometimes to consider cases where the \(\mathsf{update}\) morphisms admit more or different structure to the \(\mathsf{view}\) morphisms in \(\cat{C}\). We can generalize Bayesian lenses to such a mixed case as follows.

\begin{defn} \label{def:mixed-bayes-lens}
We first generalize the action \(\odot\). Let \(\cat{D}\) be the category of \(\mathsf{update}\) morphisms. We assume it to be \(\Cat{V}\text{-enriched}\). We define an action \(\oslash : \hat{\cat{C}} \to \Cat{V\mdash Cat}(\check{\cat{D}}, \check{\cat{D}})\) of \(\hat{\cat{C}}\) on \(\check{\cat{D}}\) as a straightforward generalization of \(\odot\) as defined in \eqref{eq:L-action}. Once again, we give only the action on objects; the action on morphisms is analogous.
\begin{equation*} %
\begin{aligned}
\oslash : \hat{\cat{C}} & \to \Cat{V\mdash Cat}(\check{\cat{D}}, \check{\cat{D}}) \\
\hat{M} & \mapsto
  \begin{pmatrix*}
    \hat{M} \oslash - & : & \check{\cat{D}} & \to & \check{\cat{D}} \\
    & & P & \mapsto & \Cat{V}\big( \hat{M}(I), P \big)
  \end{pmatrix*}
\end{aligned}
\end{equation*}
\(\oslash\) equips \(\check{\cat{D}}\) with a \(\hat{\cat{C}}\text{-actegory}\) structure, just as in Proposition \ref{prop:odot-actegory}. We define a corresponding category of \textbf{mixed Bayesian lenses} as the obvious generalization of Definition \ref{def:bayes-lens}. Objects \((\hat{X}, \check{Y})\) are pairs of a presheaf on \(\cat{C}\) and a copresheaf on \(\cat{D}\), and morphisms \((\hat{X}, \check{A}) \lensto (\hat{Y}, \check{B})\) are elements of
\begin{equation*}
\Cat{Optic}_{\times, \oslash}\Big((\hat{X}, \check{A}), (\hat{Y}, \check{B})\Big)
= \int^{\hat{M} \, : \, \hat{\cat{C}}} \hat{\cat{C}}(\hat{X}, \hat{M} \times \hat{Y}) \times \check{\cat{D}}(\hat{M} \oslash \check{B}, \check{A}) \: .
\end{equation*}
\end{defn}

\begin{ex}[State-dependent algebra homomorphisms] \label{ex:stat-alg}
Let \(\cat{C} = \Kl(M)\) be the Kleisli category of a monad \(M : \Set \to \Set\) and let \(\cat{D} = \EM(M)\) be its Eilenberg-Moore category. Both \(\cat{C}\) and \(\cat{D}\) are \(\Set\text{-enriched}\). A (representable) mixed Bayesian lens \(\optar{v}{u} : (S, T) \lensto (A, B)\) over \(\cat{C}\) and \(\cat{D}\) is then given by a Kleisli morphism \(v : S \to MA\) and an \(S\text{-state-dependent}\) algebra homomorphism \(u : \Set\big(MS, \EM(M)(B, T)\big)\). Under the forgetful functor \(U : \EM(M) \to \Set\) and by the Cartesian closed structure of \(\Set\), \(u\) is equivalently a function \(u^\flat : MS \times B \to T\) such that \(u^\flat(\mu, -) : B \to T\) is an \(M\text{-algebra}\) homomorphism for each \(\mu : MS\).
\end{ex}

\section{Bayesian updates compose optically}
\label{sec:org5f354bb}
\label{sec:bayes-compose}

The categories of state-dependent channels and of Bayesian lenses defined in \secref{sec:stat-cat} and \secref{sec:bayes-lens} are substantial generalizations of concrete Bayesian inversion as introduced in \secref{sec:comp-prob}. In this section, we concentrate on the latter, noting that every pair of a stochastic channel \(c\) and its (state-dependent) inversion \(c^\dag_{(\cdot)}\) constitutes a simple Bayesian lens \(\optar{c}{c^\dag}\) satisfying the following definition. We adopt the terminology of `exact' and `approximate' inference \citep{Knoblauch2019Generalized}.

\begin{defn}[Exact and approximate Bayesian lens]
Let \(\optar{c}{c^\dag} : (X, X) \lensto (Y, Y)\) be a simple Bayesian lens. We say that \(\optar{c}{c^\dag}\) is \textbf{exact} if \(c\) admits Bayesian inversion and, for each \(\pi : I \klto X\) such that \(c \klcirc \pi\) has non-empty support, \(c\) and \(c^\dag_\pi\) together satisfy equation \eqref{eq:bayes-abstr}. Simple Bayesian lenses that are not exact are said to be \textbf{approximate}.
\end{defn}

We seek to prove the following theorem, which is the main result of this paper.

\begin{thm} \label{thm:optical-bayes}
Let \(\optar{c}{c^\dag}\) and \(\optar{d}{d^\dag}\) be sequentially composable exact Bayesian lenses. Then the contravariant component of the composite lens \(\optar{d}{d^\dag} \lenscirc \optar{c}{c^\dag} \cong \optar{d \klcirc c}{c^\dag \circ c^\ast d^\dag}\) is, up to \(d \klcirc c \klcirc \pi \text{-almost-}\allowbreak\text{equality}\), the Bayesian inversion of \(d \klcirc c\) with respect to any state \(\pi\) on the domain of \(c\) such that \(c \klcirc \pi\) has non-empty support. That is to say, \emph{Bayesian updates compose optically}: \((d \klcirc c)^\dag_\pi \overset{d \klcirc c \klcirc \pi}{\sim} c^\dag_\pi \klcirc d^\dag_{c \klcirc \pi}\). Graphically:
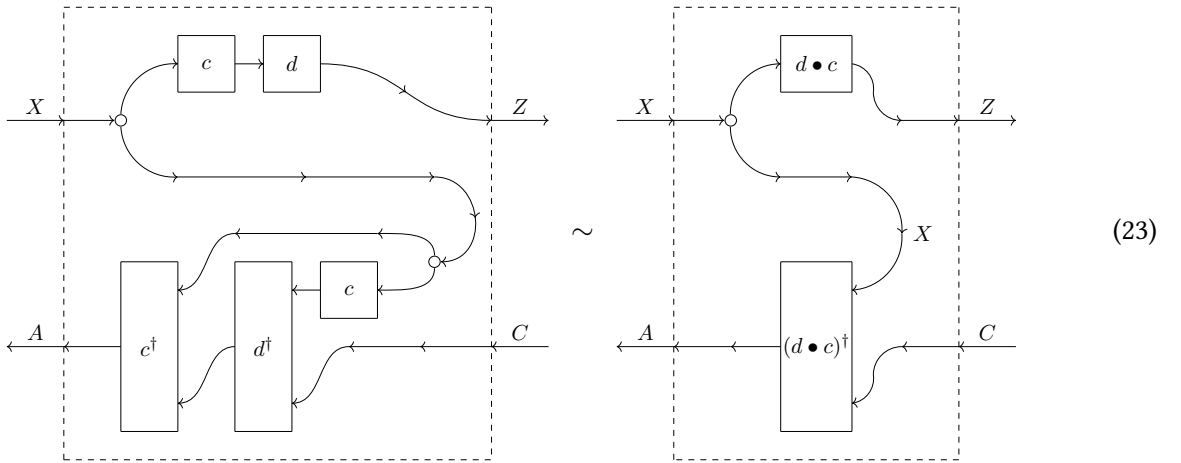
\begin{equation} \label{eq:bayes-lens-composite}
\scalebox{0.75}{\input{img/lens-cartesian-dc-2b.tikz}}
\; \sim \;
\scalebox{0.75}{\input{img/lens-cartesian-dc-0.tikz}}
\end{equation}
\end{thm}

\begin{cor} \label{cor:optical-bayes}
Let \(\cat{C}^\dag\) be the wide subcategory of channels in \(\cat{C}\) that admit Bayesian inversion (Definition \ref{def:admit-bayes}). Then \(\cat{C}^\dag\) embeds functorially into \(\Cat{BayesLens}\). On objects, the embedding is given by \(X \mapsto (\hat{X}, \check{X})\); on morphisms, \(c \mapsto \optar{c}{c^\dag}\).
\end{cor}

Because Bayesian inversion is only determined up to almost-equality, the embedding \(\cat{C}^\dag \hookrightarrow \Cat{BayesLens}\) is not unique, requiring a choice of inversion for each channel. However, in most situations of practical interest, there is a canonical choice. For those channels which have density function representations, the canonical choice is given by Proposition \ref{prop:bayes-density-graph} or equation \eqref{eq:bayes-density-krn}; alternatively, by restricting to finite support, Bayesian inversions are actually unique.

We supply proofs of Theorem \ref{thm:optical-bayes} in various copy-delete categories at various levels of abstraction, starting with the concrete case of finitely-supported probability in \(\Kl(\Dst)\). We follow this with the most abstract case, in an arbitrary copy-delete category admitting Bayesian inversion (first without and then with density functions), followed then by the case of s-finite measures (with density functions), from which we also recover the discrete result.

For pedagogical purposes, the structure of this section mirrors that of \secref{sec:comp-prob}, and we have attempted to structure the proofs to emphasize their commonalities.

\subsection{Discrete case}
\label{sec:orgebdc5b0}
\label{sec:finite-result}

In this section, we work in the category of stochastic channels \(\cat{C} = \Kl(\Dst)\), described in \secref{sec:fin-prob}. Note that with finite support, almost-equality reduces to equality, and so Bayesian inversions where they exist are unique.

\begin{proof}[Proof of Theorem \ref{thm:optical-bayes}]
Suppose \(p : X \to \Dst Y\) and \(q : Y \to \Dst Z\). Given a prior \(\rho : 1 \to \Dst X\) on \(X\), we are interested in the Bayesian inversion \((q \bullet p)^\dag_\rho : Z \to \Dst X\) of \(q \bullet p : X \to \Dst Z\) with respect to \(\rho\). Following \eqref{eq:finite-bayes}, we have
\begin{equation*}
\begin{aligned}
\bdag{(q \klcirc p)} &: \Dst X \times Z \to \Dst X \\
&:= \rho \times z \mapsto \sum_{x : X} \, \boxed{\frac{(q \klcirc p)(z|x) \cdot \rho(x)}{\sum_{x':X} (q \klcirc p)(z | x') \cdot \rho(x')}} \, \ket{x}
= \sum_{x : X} \, \boxed{\frac{(q \klcirc p)(z|x) \cdot \rho(x)}{(q \klcirc p \klcirc \rho)(z)}} \, \ket{x}.
\end{aligned}
\end{equation*}
The lens composite of \(q^\dag\) and \(p^\dag\) with respect to \(\rho\) is \(p^\dag_\rho \klcirc q^\dag_{p \klcirc \rho}\). Our task is therefore to show that
\begin{equation*}
p^\dag_\rho \klcirc q^\dag_{p \klcirc \rho}(z) = \sum_{x : X} \, \boxed{\frac{(q \klcirc p)(z|x) \cdot \rho(x)}{(q \klcirc p \klcirc \rho)(z)}} \, \ket{x} = (q \klcirc p)^\dag_\rho(z) \, .
\end{equation*}
By Kleisli extension \eqref{eq:klext-finite}, for any \(\sigma : 1 \to \Dst Y\),
\begin{equation*}
p^\dag_\rho \klcirc \sigma = \sum_{x : X} \, \boxed{\sum_{y : Y} p^\dag_\rho(x|y) \cdot \sigma(y)} \, \ket{x} = \sum_{x : X} \, \boxed{\sum_{y : Y} \left( \frac{p(y|x) \cdot \rho(x)}{(p \klcirc \rho)(y)} \right) \sigma(y)} \, \ket{x} .
\end{equation*}
Now, let \(\sigma \mapsto q^\dag_{p \klcirc \rho}(z)\), so
\begin{align*}
p^\dag_\rho \klcirc q^\dag_{p \klcirc \rho}(z)
&= \sum_{x : X} \, \boxed{\sum_{y : Y} \left( \frac{p(y|x) \cdot \rho(x)}{(p \klcirc \rho)(y)} \right) \cdot (q^\dag_{p \klcirc \rho})(y|z)} \, \ket{x} \\
&= \sum_{x:X} \, \boxed{\sum_{y:Y} \left( \frac{p(y|x) \cdot \rho(x)}{(p \klcirc \rho)(y)} \right) \cdot \left( \frac{q(z|y) \cdot (p \klcirc \rho)(y)}{(q \klcirc p \klcirc \rho)(z)} \right)} \, \ket{x} \\
&= \sum_{x:X} \, \boxed{\sum_{y:Y} \frac{q(z|y) \cdot p(y|x) \cdot \rho(x)}{(q \klcirc p \klcirc \rho)(z)}} \, \ket{x} \\
&= \sum_{x:X} \, \boxed{\frac{(q \klcirc p)(z|x) \cdot \rho(x)}{(q \klcirc p \klcirc \rho)(z)}} \, \ket{x} \\
&= (q \klcirc p)^\dag_\rho(z)
\end{align*}
as required.
\end{proof}

\subsection{Abstract case}
\label{sec:orgb1e9d1c}

Here, we work in an arbitrary copy-delete category \(\cat{C}\) with those morphisms that admit Bayesian inversion in the abstract sense of equation \eqref{eq:bayes-abstr} (\secref{sec:abstract-bayes}). This proof implies the result in the more concrete categories \(\Kl(\Dst)\) and \(\Cat{sfKrn}\); we include those for their computational and pedagogical content.

\begin{proof}[Proof of Theorem \ref{thm:optical-bayes}]
Suppose \(c^\dag_\pi : Y \klto X\) is the Bayesian inverse of \(c : X \klto Y\) with respect to \(\pi : I \klto X\). Suppose also that \(d^\dag_{c \klcirc \pi} : Z \klto Y\) is the Bayesian inverse of \(d : Y \klto X\) with respect to \(c \klcirc \pi : I \klto Y\), and that \((d \klcirc c)^\dag_\pi : Z \klto X\) is the Bayesian inverse of \(d \klcirc c : X \klto Z\) with respect to \(\pi : I \klto X\):
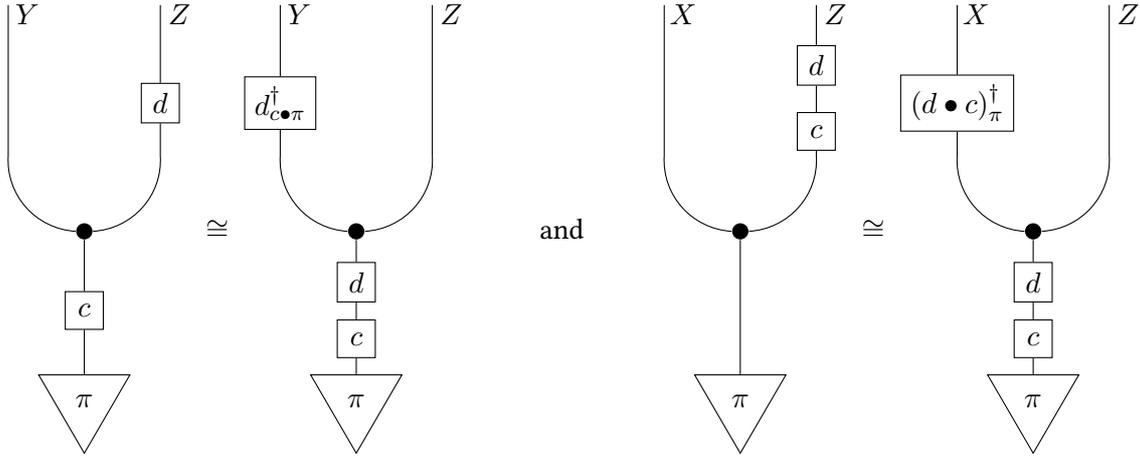
\[
\input{img/joint-d-c-pi.tikz} \ \cong \ \input{img/joint-ddag-dc-pi.tikz}
\hspace{0.06\linewidth}
\text{and}
\hspace{0.06\linewidth}
\input{img/joint-dc-pi.tikz} \ \cong \ \input{img/joint-dcdag-dc-pi.tikz}
\]

The lens composite of these Bayesian inverses has the form \(c^\dag_\pi \klcirc d^\dag_{c \klcirc \pi} : Z \klto X\), so to establish the result it suffices to show that
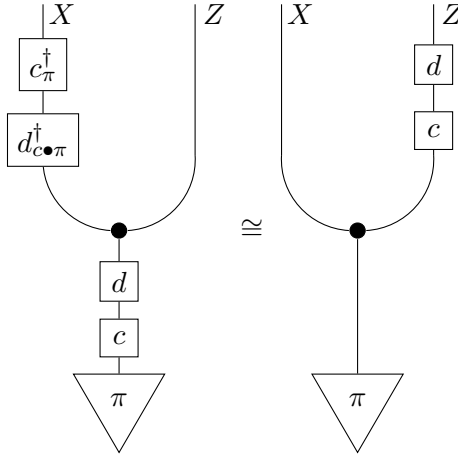
\[ \input{img/joint-cdag-ddag-dc-pi.tikz} \ \cong \ \input{img/joint-dc-pi.tikz} \]
where we can think of the left-hand side as `unfolding' along the residual the left-hand side of \eqref{eq:bayes-lens-composite}.

We have the following isomorphisms:
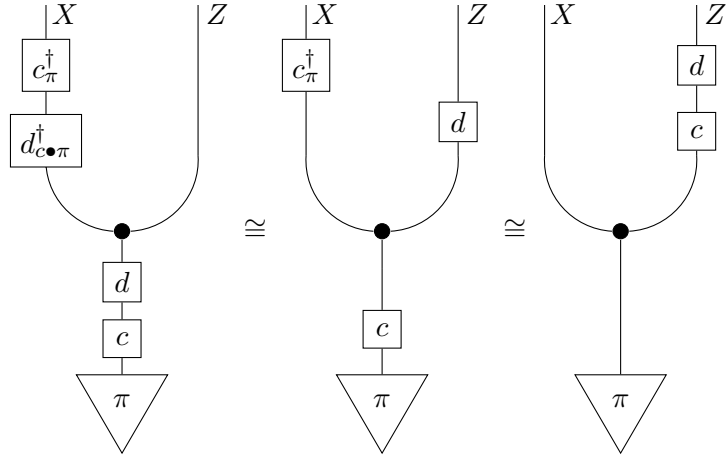
\[ \input{img/joint-cdag-ddag-dc-pi.tikz} \ \cong \ \input{img/joint-cdag_d-c-pi.tikz} \ 
\cong \ \input{img/joint-dc-pi.tikz} \]
where the first obtains because \(d^\dag_{c \klcirc \pi}\) is the Bayesian inverse of \(d\) with respect to \(c \klcirc \pi\), and the second because \(c^\dag_\pi\) is the Bayesian inverse of \(c\) with respect to \(\pi\). Hence, \(c^\dag_\pi \klcirc d^\dag_{c \klcirc \pi}\) and \((d \klcirc c)^\dag_\pi\) are both Bayesian inversions of \(d \klcirc c\) with respect to \(\pi\). Since Bayesian inversions are almost-equal (Proposition \ref{prop:bayes-almost-equal}), we have \(c^\dag_\pi \klcirc d^\dag_{c \klcirc \pi} \overset{d \klcirc c \klcirc \pi}{\sim} (d \klcirc c)^\dag_\pi\), as required.
\end{proof}

\subsubsection{With density functions}
\label{sec:org19b04b0}
\label{sec:abstr-density-proof}

Here, we work in an abstract copy-delete category \(\cat{C}\) in which stochastic channels can be represented by effects in the sense of Definition \ref{def:density-functions} (\secref{sec:density-functions}).

\begin{proof}[Proof of Theorem \ref{thm:optical-bayes}]
Suppose \(c : X \klto Y\) and \(d : Y \klto Z\) are represented by effects
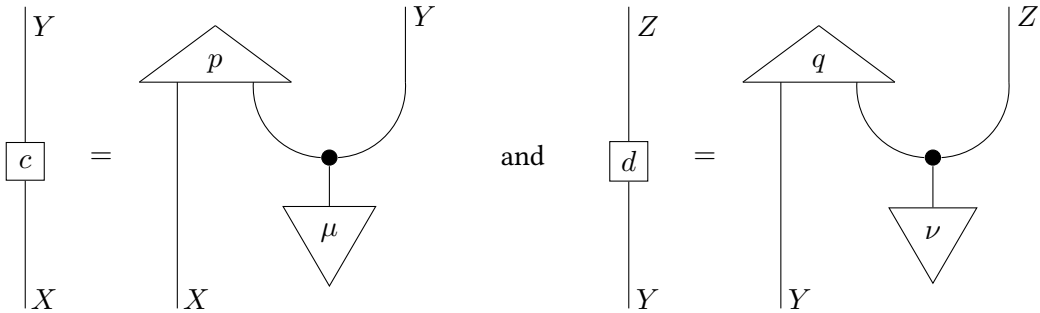
\[
\input{img/def-density-function-c.tikz}
\hspace{0.05\linewidth}
\text{and}
\hspace{0.05\linewidth}
\input{img/def-density-function-d.tikz}
\]
so that the composite \(d \klcirc c : X \klto Z\) is given by
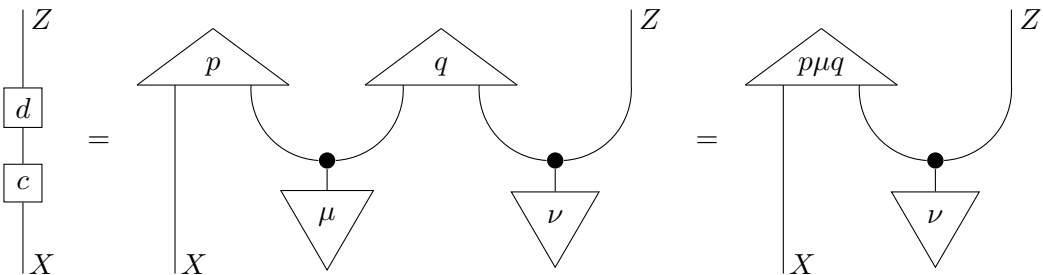
\[ \input{img/def-density-function-dc.tikz} \]
where the effect \(p \mu q : X \otimes Z \klto I\) is defined in the obvious way.

Following Proposition \ref{prop:bayes-density-graph}, the Bayesian inverse \(c^\dag_\pi : Y \klto X\) of \(c\) with respect to \(\pi : I \klto X\) is given by
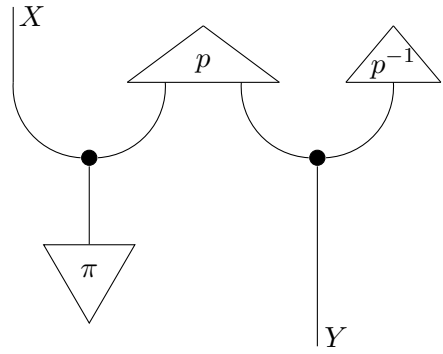
\[ \input{img/channel-density-cdag-pi.tikz} \]
where \(p^{-1} : Y \klto I\) is a \(\mu\text{-almost-inverse}\) for the effect
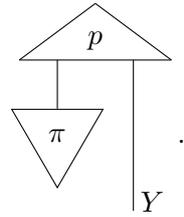
\[ \input{img/likelihood-p-pi.tikz} \, . \]

Similarly, the Bayesian inverse \(d^\dag_{c \klcirc \pi} : Z \klto Y\) of \(d\) with respect to \(c \klcirc \pi : I \klto Y\) is
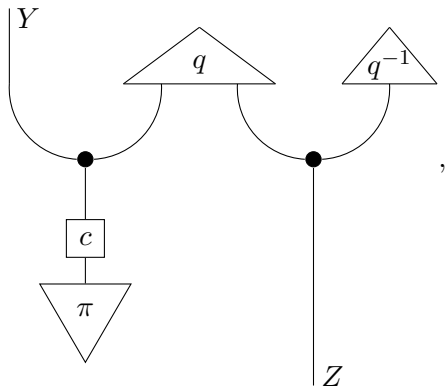
\[ \input{img/channel-density-ddag-c-pi.tikz} , \]
with \(q^{-1} : Z \klto I\) the corresponding \(\nu\text{-almost-inverse}\) for 
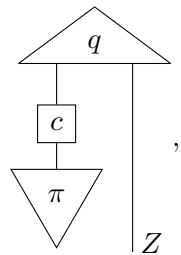
\[ \input{img/likelihood-q-c-pi.tikz} , \]
and the Bayesian inverse for \(d \klcirc c\) with respect to \(\pi\) is \((d \klcirc c)^\dag_\pi : Z \klto X\), with the form
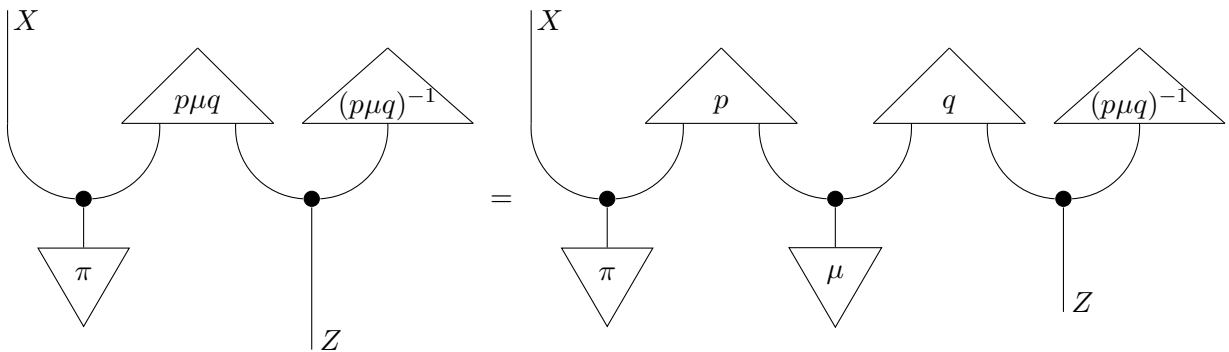
\[ \input{img/channel-density-dcdag-pi-simpl.tikz} \ =\ \input{img/channel-density-dcdag-pi.tikz} \]
where \((p \mu q)^{-1}\) is also a \(\nu\text{-almost-inverse}\) for \(q \klcirc \big((c \klcirc \pi) \otimes \id_Z\big) : Z \klto I\).

We seek to show that \((d \klcirc c)^\dag_\pi \overset{d \klcirc c \klcirc \pi}{\sim} c^\dag_\pi \klcirc d^\dag_{c \klcirc \pi}\). We start from the lens composite \(c^\dag_\pi \klcirc d^\dag_{c \klcirc \pi}\) which is given by
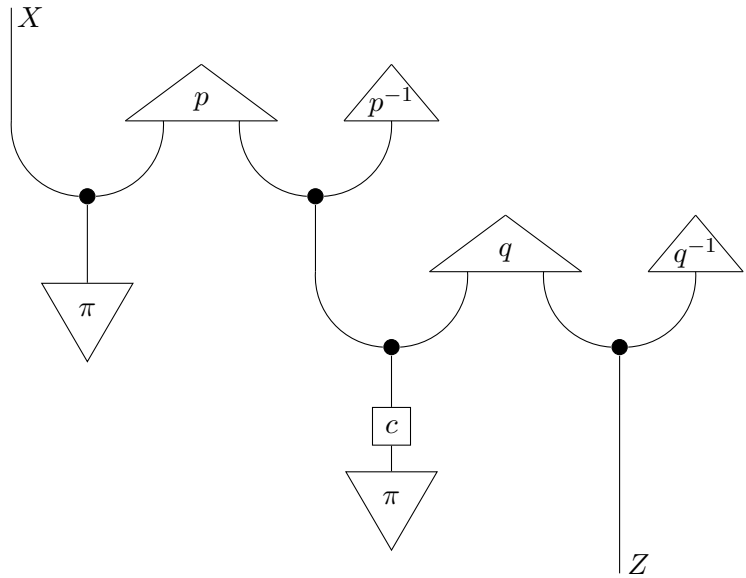
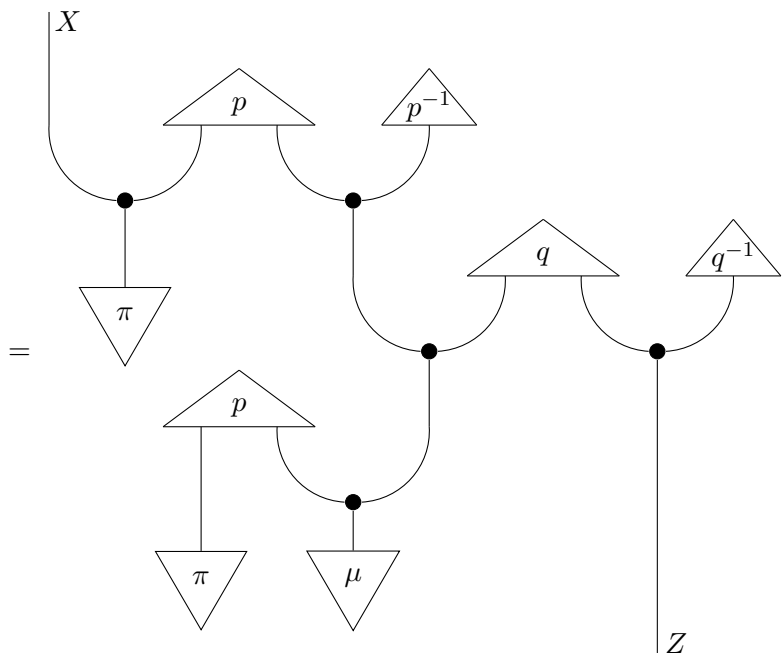
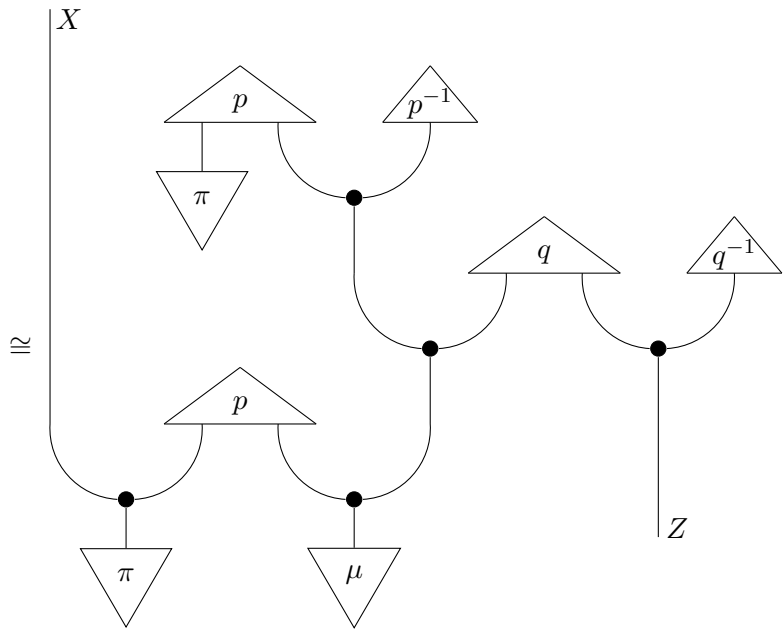
\[ \input{img/channel-density-cdag-pi-ddag-c-pi-1.tikz} \]
\begin{align*}
&\ =\ \input{img/channel-density-cdag-pi-ddag-c-pi-2.tikz} \\
&\hspace{0.033\linewidth}\big(\text{by definition of } c \big)
\end{align*}
\begin{align*}
&\ \cong\ \input{img/channel-density-cdag-pi-ddag-c-pi-3.tikz} \\
&\hspace{0.033\linewidth}\big(\text{by associativity \eqref{eq:comonoid-law} and commutativity \eqref{eq:comonoid-commute} of copiers}\big) \\
\end{align*}
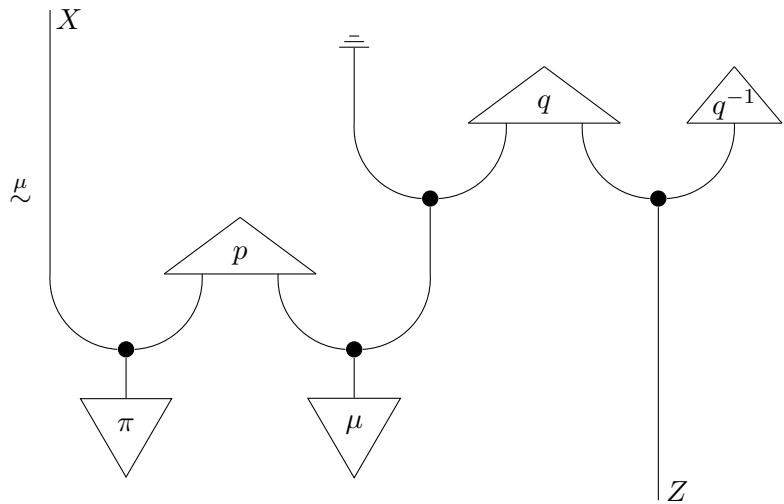
\begin{align*}
&\ \overset{\mu}{\sim}\ \input{img/channel-density-cdag-pi-ddag-c-pi-4.tikz} \\
&\hspace{0.033\linewidth}\big(\text{by definition of } p^{-1} \big) \\
\end{align*}
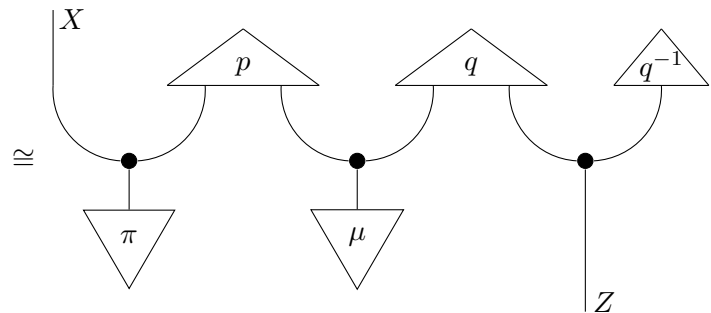
\begin{align*}
&\ \cong\ \input{img/channel-density-cdag-pi-ddag-c-pi-5.tikz} \\
&\hspace{0.033\linewidth}\big(\text{by unitality \eqref{eq:comonoid-law} of copiers}\big) \\
\end{align*}
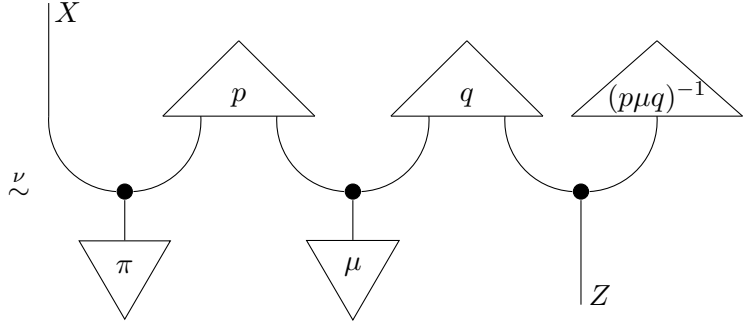
\[ \overset{\nu}{\sim}\ \input{img/channel-density-dcdag-pi.tikz} \, . \]
The last line follows by Lemma \ref{lemma:eff-chan-almost-eq}, since the two almost-inverses \((p \mu q)^{-1}\) and \(q^{-1}\) are \(\nu\text{-almost}\) equal (Proposition \ref{prop:almost-inverse-almost-equal}).

We have shown that \((d \klcirc c)^\dag_\pi \overset{\nu}{\sim} c^\dag_\pi \klcirc d^\dag_{c \klcirc \pi}\). Recall that \(d\) is represented by an effect with respect to the state \(\nu\). So by Lemma \ref{lemma:eff-blocks-almost-eq}, we have \((d \klcirc c)^\dag_\pi \overset{d \klcirc c \klcirc \pi}{\sim} c^\dag_\pi \klcirc d^\dag_{c \klcirc \pi}\), as required.
\end{proof}

\subsection{S-Finite case with density functions}
\label{sec:orgdb85741}

Here, we instantiate the abstract density function proof of \secref{sec:abstr-density-proof} in the category \(\Cat{sfKrn}\) of s-finite kernels described in \secref{sec:sfKrn}, in order to obtain a form of the result commensurate with the informal form of Bayes' rule \eqref{eq:bayes-density-informal}. Then, restricting to finitely supported measures, we recover the discrete case of the result in \(\Kl(\Dst)\).

\begin{proof}[Proof of Theorem \ref{thm:optical-bayes}]
Equation \eqref{eq:bayes-density-krn} states that, by interpreting the string diagram of Proposition \ref{prop:bayes-density-graph} for Bayesian inversion via density functions in \(\Cat{sfKrn}\), the Bayesian inverse \(d^\dag_\rho\) of \(d : Y \klto Z\) with respect to \(\rho : I \klto Y\) is given by
\begin{align*}
d^\dag_\rho : Z \times \Sigma_Y \to [0, \infty]
:= z \times B & \mapsto \left( \int_{y:B} \rho(\d y) \, q(z|y) \right) q^{-1}(z) \\
&= q^{-1}(z) \int_{y:B} q(z|y) \, \rho(\d y)
\end{align*}
where \(q^{-1} : Z \klto I\) is a \(\nu\text{-almost-inverse}\) for \(q \klcirc ((c \klcirc \pi) \, \otimes \id_Z)\), given up to \(\nu\text{-almost-equality}\) by
\[ q^{-1} : Z \to [0, \infty] := z \mapsto \left( \int_{y:Y} q(z|y) \, \mu(\d y) \int_{x:X} p(y|x) \, \pi(\d x) \right)^{-1} . \]

Suppose then that
\begin{equation*}
\rho = c \klcirc \pi : 1 \times \Sigma_Y \to [0, \infty]
:= \ast \times B \mapsto \int_{y:B} \mu(\d y) \int_{x:X} p(y|x) \, \pi(\d x) \, .
\end{equation*}
We therefore have, by direct substitution,
\begin{align*}
d^\dag_{c \klcirc \pi} : Z \times \Sigma_Y \to [0, \infty]
:= z \times B & \mapsto q^{-1}(z) \int_{y:B} q(z|y) \mu(\d y) \int_{x:X} p(y|x) \, \pi(\d x) .
\end{align*}

We now write down directly the Bayesian inverse of the composite channel, \((d \klcirc c)^\dag_\pi : Z \klto X\):
\begin{align*}
(d \klcirc c)^\dag_\pi : Z \times \Sigma_X \to & [0, \infty] \\
 := z \times A \mapsto & \left( \int_{x:A} \pi(\d x) \, (p \mu q)(z|x) \right) (p \mu q)^{-1}(z) \\
  = & (p \mu q)^{-1}(z) \int_{x:A} (p \mu q)(z|x) \, \pi(\d x) \\
  = & (p \mu q)^{-1}(z) \int_{x:A} \int_{y:Y} q(z|y) \, \mu(\d y) \, p(y|x) \, \pi(\d x)
\end{align*}
where \((p \mu q)^{-1}\) is a \(\nu\text{-almost-inverse}\) for \((p \mu q) \klcirc (\pi \otimes \id_Z)\), or equivalently (since almost-inverses are almost-equal; Proposition \ref{prop:almost-inverse-almost-equal}), a \(\nu\text{-almost-inverse}\) for \(q \klcirc ((c \klcirc \pi) \, \otimes \id_Z)\).

The lens form of the inverse composite is \(c^\dag_\pi \klcirc d^\dag_{c \klcirc \pi} : Z \klto X\), and we follow the diagrammatic reasoning to show that this is equal to the direct form above:
\begin{align*}
c^\dag_\pi \klcirc d^\dag_{c \klcirc \pi} : Z \times \Sigma_X \to & [0, \infty] \\
 := z \times A \mapsto
    & \, \int_{y:Y} c^\dag_\pi(A|y) \cdot (d^\dag_{c \klcirc \pi})(\d y|z) \\
  = & \, \int_{y:Y} \left( p^{-1}(y) \, \int_{x:A} p(y|x) \, \pi(\d x) \right) \cdot \left( q^{-1}(z) \, q(z|y) \, \mu(\d y) \, \int_{x:X} p(y|x) \, \pi(\d x) \right) \\
  = & \, \int_{y:Y} \left( p^{-1}(y) \, \int_{x:X} p(y|x) \, \pi(\d x) \right) \cdot \left( q^{-1}(z) \, q(z|y) \, \mu(\d y) \, \int_{x:A} p(y|x) \, \pi(\d x) \right) \\
  \overset{\mu}{\sim} & \, q^{-1}(z) \int_{y:Y} q(z|y) \, \mu(\d y) \int_{x:A} p(y|x) \, \pi(\d x) \\
  = & \, q^{-1}(z) \int_{x:A} \int_{y:Y} q(z|y) \, \mu(\d y) \, p(y|x) \, \pi(\d x) \\
  \overset{\nu}{\sim} & \, \big((d \klcirc c)^\dag_\pi\big)(A|z)
\end{align*}
The second equality follows by Fubini's theorem in \(\Cat{sfKrn}\) (equivalently, by the symmetry of the monoidal product \(\otimes\); see \secref{sec:sfKrn}), and the last line by Lemma \ref{lemma:eff-chan-almost-eq} since almost-inverses are almost-equal (Proposition \ref{prop:almost-inverse-almost-equal}); compare the final part of the graphical proof in \secref{sec:abstr-density-proof}. Finally, by Lemma \ref{lemma:eff-blocks-almost-eq}, we have \((d \klcirc c)^\dag_\pi \overset{d \klcirc c \klcirc \pi}{\sim} c^\dag_\pi \klcirc d^\dag_{c \klcirc \pi}\), also as in \secref{sec:abstr-density-proof}.
\end{proof}

\begin{cor}
By restricting to finitely supported measures, we recover the discrete-case result of \secref{sec:finite-result}.
\begin{proof}
Suppose then that the prior \(\pi : I \klto X\) is represented by a density function \(\rho : X \klto I\) with respect to a measure \(\upsilon : I \klto X\). Suppose further that the measures \(\mu\), \(\nu\), and \(\upsilon\) are finitely supported. Then
\begin{align*}
\big((d \klcirc c)^\dag_\pi\big)({x} \, | \, z)
&= q^{-1}(z) \int_{x:\{x\}} \int_{y:Y} q(z|y) \, \mu(\d y) \, p(y|x) \, \upsilon(\d x) \, \rho(x) \\
&= \frac{\int_{x:\{x\}} \int_{y:Y} q(z|y) \, \mu(\d y) \, p(y|x) \, \upsilon(\d x) \, \rho(x)}{\int_{y':Y} q(z|y') \, \mu(\d y') \int_{x':X} p(y'|x') \, \upsilon(\d x) \, \rho(x')} \\
&= \frac{\sum_{y:Y} q(z|y) \, p(y|x) \, \rho(x)}{\sum_{y':Y} q(z|y') \sum_{x':X} p(y'|x') \, \rho(x')}
\end{align*}
Now, assume that the density functions are indeed probability densities (\emph{i.e.}, they sum to 1 over their support), and note that they are therefore functions of the form \(X \times Y \to [0, 1]\). We therefore recognise them as stochastic matrices \(X \klto Y\) in \(\Kl(\Dst)\), and can interpret the foregoing expression as
\[
\frac{\sum_{y:Y} q(z|y) \, p(y|x) \, \rho(x)}{\sum_{y':Y} q(z|y') \sum_{x':X} p(y'|x') \, \rho(x')} = \frac{\big(q \klcirc p\big)(z|x) \cdot \rho(x)}{\big(q \klcirc p \klcirc \rho\big)(z)} ,
\]
where \(\klcirc\) is now composition in \(\Kl(\Dst)\), thereby recovering the result in the discrete case.
\end{proof}
\end{cor}

\section{Lawfulness of Bayesian lenses}
\label{sec:orgbd8ea72}
\label{sec:lawfulness}

The study of Cartesian lenses substantially originates in the context of bidirectional transformations of data in the computer science and database community \citep{Foster2007Combinators,Riley2018Categories}, where we can think of the \(\mathsf{view}\) (or \(\mathsf{get}\)) function as returning part of a database record, and the \(\mathsf{update}\) (or \(\mathsf{put}\)) function as `putting' a part into a record and returning the updated record. In this setting, extra structure known as \emph{lens laws} can be imposed on lenses to ensure that they are `well-behaved' with respect to database behaviour. Well-behaved and `very well-behaved' lenses in the database context roughly correspond to our notion of `exact' Bayesian lens, but as we will see, even exact Bayesian lenses are only weakly lawful in the database sense.

We will concentrate on the three lens laws that have attracted recent study \citep{Riley2018Categories,Boisseau2020String}: \texttt{GetPut}, \texttt{PutGet}, and \texttt{PutPut}. A Cartesian lens satisfying the former two is \emph{well-behaved} while a lens satisfying all three is \emph{very well-behaved}, in the terminology of \citet{Foster2007Combinators}. Informally, \texttt{GetPut} says that getting part of a record and putting it straight back returns an unchanged record; \texttt{PutGet} says that putting a part into a record and then getting it returns the same part that we started with; and \texttt{PutPut} says that putting one part and then putting a second part has the same effect on a record as just putting the second part (that is, \(\mathsf{update}\) completely overwrites the part in the record). We will express these laws graphically, and consider them each briefly in turn.

Note first that we can lift any channel \(c\) in the base category \(\cat{C}\) into any state-dependent fibre \(\Fun{Stat}(A)\) using the constant (identity-on-objects) functor taking \(c\) to the constant-valued state-indexed channel \(\rho \mapsto c\) that maps any state \(\rho\) to \(c\). We can lift diagrams in \(\cat{C}\) into the fibres accordingly.

\paragraph{\texttt{GetPut}}
\label{sec:org71eb4b7}

\begin{defn}
A lens \(\optar{c}{c^\dag}\) is said to satisfy the \texttt{GetPut} law if it satisfies the left isomorphism in \eqref{eq:getput-law} below. Equivalently, because the copier induced by the Cartesian product is natural (\emph{i.e.}, \(\copier \circ f \cong (f \times f) \circ \copier\)), for any state \(\pi\), we say that \(\optar{c}{c^\dag}\) satisfies \texttt{GetPut} with respect to \(\pi\) if it satisfies the right isomorphism in \eqref{eq:getput-law} below.
\begin{equation} \label{eq:getput-law}
\input{img/bayes-getput1.tikz}
\hspace{0.06\linewidth}
\Longrightarrow
\hspace{0.06\linewidth}
\input{img/bayes-getput2.tikz}
\end{equation}
\end{defn}

\begin{prop}
When \(c\) is causal, the exact Bayesian lens \(\optar{c}{c^\dag}\) satisfies the \texttt{GetPut} law with respect to any state \(\pi\) for which \(c\) admits Bayesian inversion.
\begin{proof}
Starting from the right-hand-side of \eqref{eq:getput-law}, we have the following chain of isomorphisms
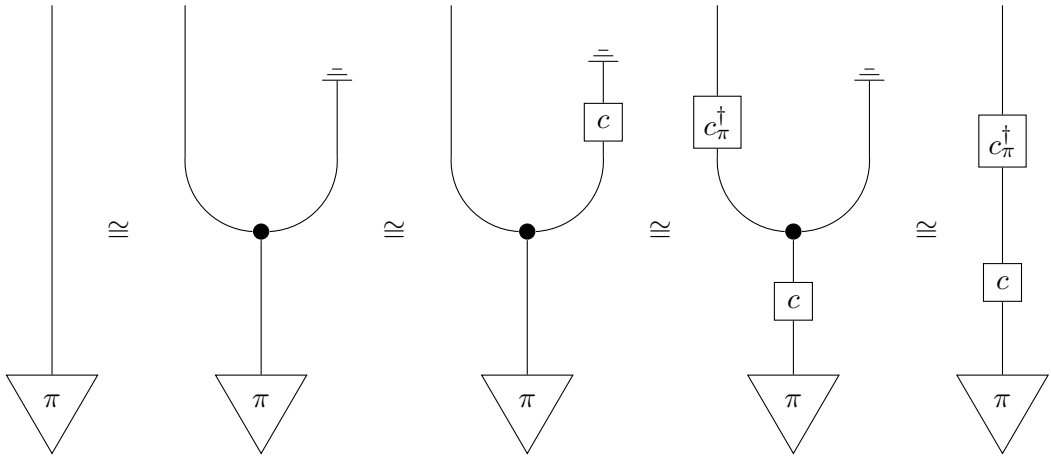
\[ \input{img/bayes-getput3.tikz} \]
where the first holds by the unitality of \(\bcopier\) \eqref{eq:comonoid-law}, the second by the causality of \(c\), the third since \(c\) admits Bayesian inversion \eqref{eq:bayes-abstr} with respect to \(\pi\), and the fourth again by unitality of \(\bcopier\).
\end{proof}
Note that by Bayes' law, exact Bayesian lenses only satisfy \texttt{GetPut} with respect to states. This result means that, if we think of \(c\) as generating a prediction \(c \klcirc \pi\) from a prior belief \(\pi\), then if our observation exactly matches the prediction, updating the prior \(\pi\) according to Bayes' rule results in no change.
\end{prop}

\paragraph{\texttt{PutGet}}
\label{sec:orgde44634}

The \texttt{PutGet} law is characterized for a lens \(\optar{v}{u}\) by the following isomorphism:
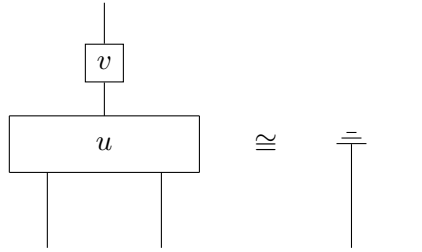
\[ \input{img/bayes-putget.tikz} \]
In general, \texttt{PutGet} does not hold for exact Bayesian lenses \(\optar{v}{u} = \optar{c}{c^\dag}\). However, because \texttt{GetPut} holds with respect to states \(\pi\), we do have \(c \circ c^\dag \circ (\pi \times c \klcirc \pi) \cong (\ground \times \id) \circ (\pi \times c \klcirc \pi)\); that is, \texttt{PutGet} holds for exact Bayesian lenses \(\optar{c}{c^\dag}\) for the input \(\pi \times c \klcirc \pi\).

The reason \texttt{PutGet} fails to hold in general is that Bayesian updating by construction mixes information from the prior and the observation, according to the strength of belief. Consequently, updating a belief according to an observed state and then producing a new prediction need not result in the same state as observed; unless, of course, the prediction already matches the observation.

\paragraph{\texttt{PutPut}}
\label{sec:org6c0fef6}

Finally, the \texttt{PutPut} law for a lens \(\optar{v}{u}\) is characterized by the following isomorphism:
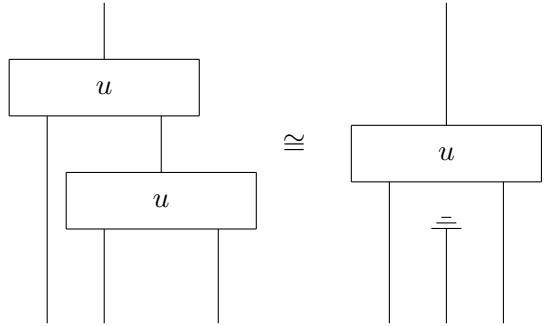
\[ \input{img/bayes-putput.tikz} \]
\texttt{PutPut} fails to hold for exact Bayesian lenses for the same reason that \texttt{PutGet} fails to hold in general: updates mix old and new beliefs, rather than entirely replace the old with the new.

\paragraph{Comment}
\label{sec:org773224b}

The lens laws were originally defined in the context of computer databases, where there is assumed to be no uncertainty: database logic is Boolean (database instances are objects in a presheaf topos \citep{Fong2018Seven}), so a `belief' is either true or false. Consequently, there can be no `mixing' of beliefs; and in database applications, such mixing may be highly undesirable (putting epistemological concerns aside). Bayesian lenses, on the other hand, live in a fuzzier world: the present author's interest in Bayesian lenses originates in their application to describing cognitive and cybernetic processes such as perception and action, and here the ability to mix beliefs according to uncertainty is desirable.

Possibly it would be of interest to give analogous information-theoretic lens laws that characterize exact and approximate Bayesian lenses and their generalizations; and we might then expect the `Boolean' lens laws to emerge in the extremal case where there is no uncertainty and only Dirac states. We leave such an endeavour for future work: Bayes' law \eqref{eq:bayes-abstr} is sufficiently concise and productive for our purposes here.

\section{References}
\label{sec:org6076cdf}
\printbibliography[heading=none]

\appendix
\section{Extraneous proofs}
\label{sec:org75fecef}
\subsection{Proposition \ref{prop:almost-inverse-almost-equal}: Almost-inverses are almost-equal}
\label{sec:org0826ef9}
\label{sec:proof:almost-inverse-almost-equal}
\begin{proof}
By assumption, we have
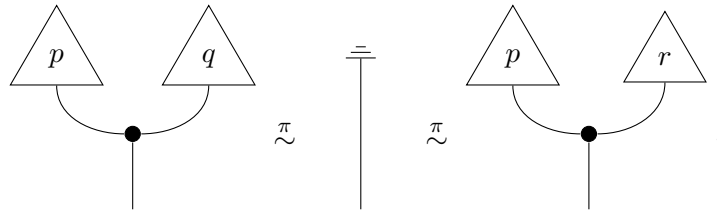
\[ \input{img/almost-inv-almost-eq-1.tikz} \; . \]
Then, by the definition of almost-equality (Definition \ref{def:almost-eq},:
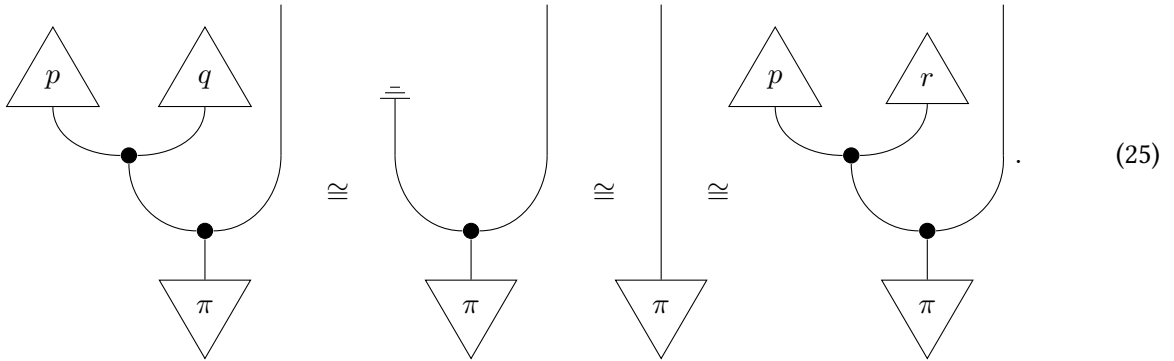
\begin{equation} \label{eq:proof:almost-inverse-almost-equal-1}
\input{img/almost-inv-almost-eq-2.tikz} \; .
\end{equation}
We seek to show that
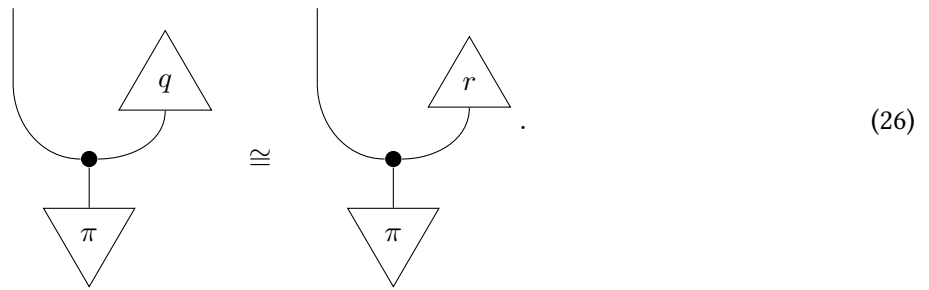
\begin{equation} \label{eq:proof:almost-inverse-almost-equal-2}
\input{img/almost-inv-almost-eq-3.tikz} \; .
\end{equation}
Substituting the right-hand-side of \eqref{eq:proof:almost-inverse-almost-equal-1} for $\pi$ in the left-hand-side of \eqref{eq:proof:almost-inverse-almost-equal-2}, we have that
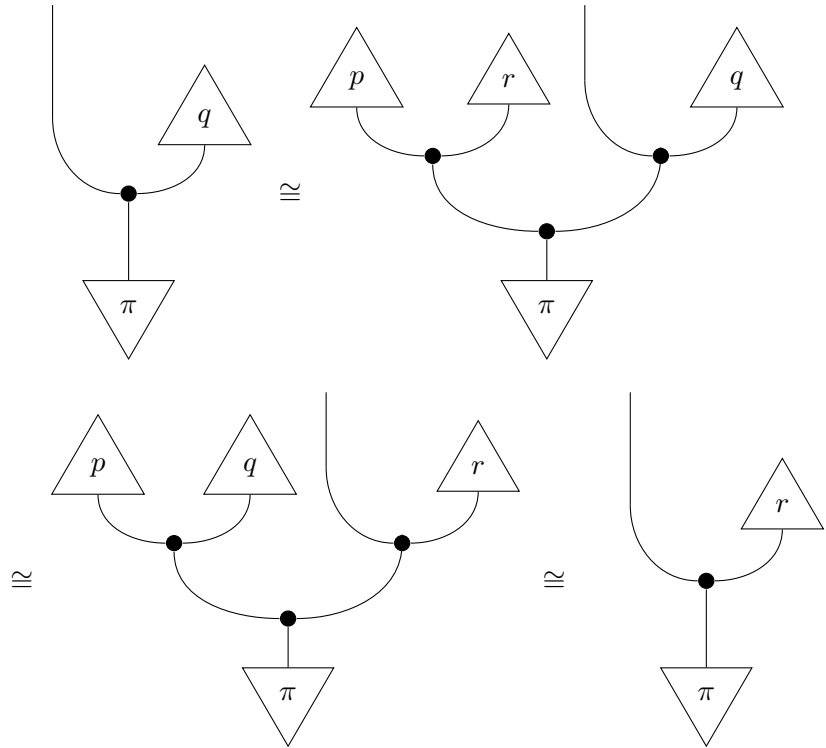
\[ \input{img/almost-inv-almost-eq-4.tikz} \]
\[ \input{img/almost-inv-almost-eq-5.tikz} \]
which establishes the result. The second isomorphism follows by the associativity of $\bcopier$ \eqref{eq:comonoid-law}, and the third \emph{ex hypothesi} and by the unitality of $\bcopier$ \eqref{eq:comonoid-law}.
\end{proof}

\subsection{Proposition \ref{prop:bayes-density-graph}: Bayesian inversion via density functions}
\label{sec:org78e8373}
\label{sec:proof:bayes-density-graph}
\begin{proof}
We seek to establish the relation \eqref{eq:bayes-abstr} characterizing Bayesian inversion. By substituting the density function representations for $c$ and $c^\dag_\pi$ into the right-hand-side of \eqref{eq:bayes-abstr}, we have
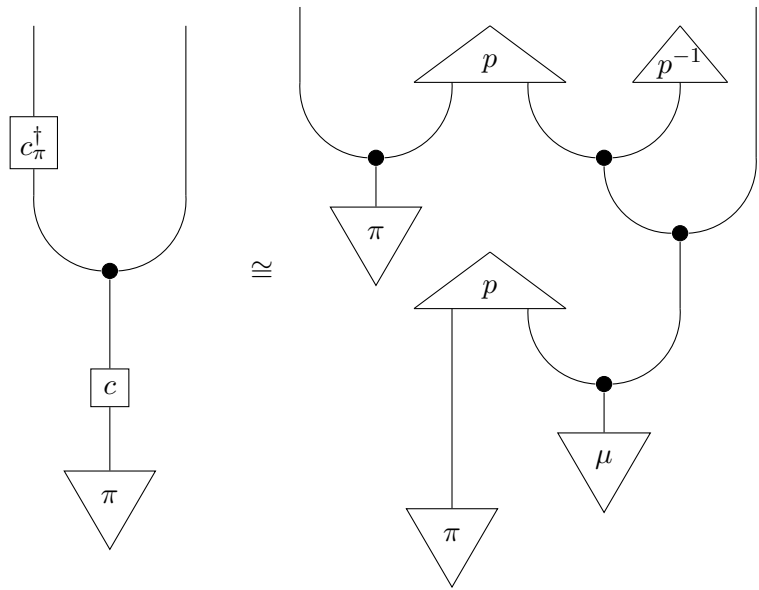
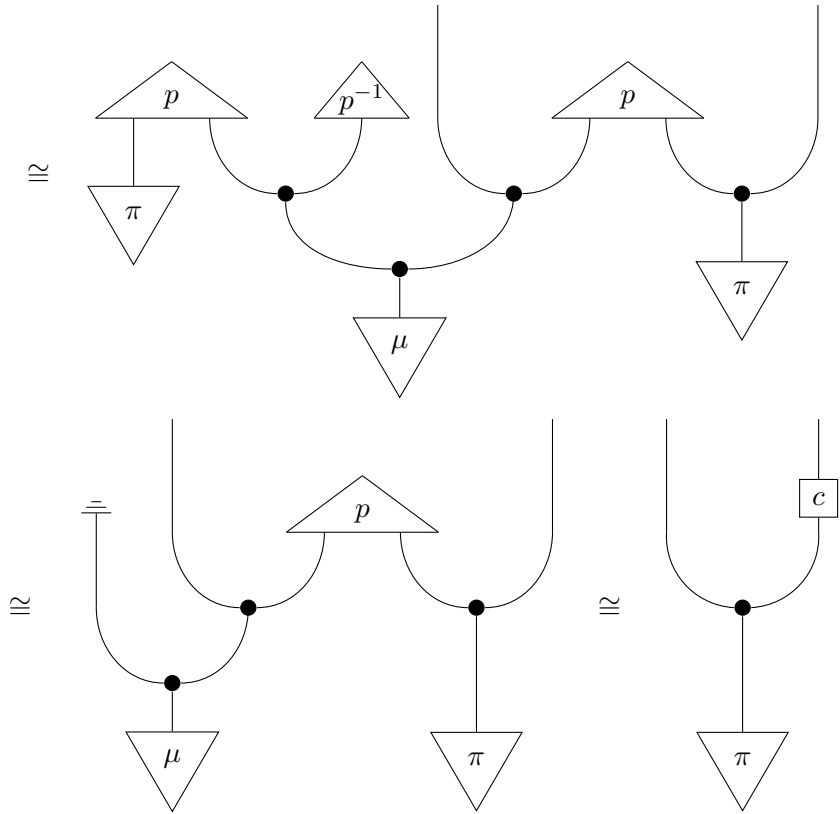
\[ \input{img/bayes-density-graph-1.tikz} \]
\[ \input{img/bayes-density-graph-2.tikz} \]
\[ \input{img/bayes-density-graph-3.tikz} \]
as required. The second isomorphism holds by the associativity of $\bcopier$ \eqref{eq:comonoid-law}, the third since $p^{-1}$ is an almost-inverse \emph{ex hypothesi}, and the fourth by the unitality of $\bcopier$ \eqref{eq:comonoid-law} and the density function representation of $c$.
\end{proof}

\subsection{Lemma \ref{lemma:eff-chan-almost-eq}}
\label{sec:org2fd94ef}
\label{sec:proof:eff-chan-almost-eq}
\begin{proof}
By assumption, we have
\[
\input{img/eff-chan-almost-eq-1.tikz}
\hspace{0.06\linewidth}
\text{and}
\hspace{0.06\linewidth}
\input{img/eff-chan-almost-eq-2.tikz}
\; .
\]
Consequently,
\[ \input{img/eff-chan-almost-eq-3.tikz} \]
\[ \input{img/eff-chan-almost-eq-4.tikz} \]
and so, by the commutativity of $\bcopier$ \eqref{eq:comonoid-commute}, $\alpha \overset{\mu}{\sim} \beta$. The first isomorphism holds by the definition of $\alpha$, the second by the associativity of $\bcopier$ \eqref{eq:comonoid-law}, the third since $q \overset{\mu}{\sim} r$, the fourth by the associativity of $\bcopier$, and the fifth by the definition of $\beta$.
\end{proof}

\subsection{Lemma \ref{lemma:eff-blocks-almost-eq}}
\label{sec:org905e17a}
\label{sec:proof:eff-blocks-almost-eq}
\begin{proof}
We start from the left-hand-side of the relation defining almost-equality (Definition \ref{def:almost-eq}), substituting the density function representation for $d$. This gives the following chain of isomorphisms:
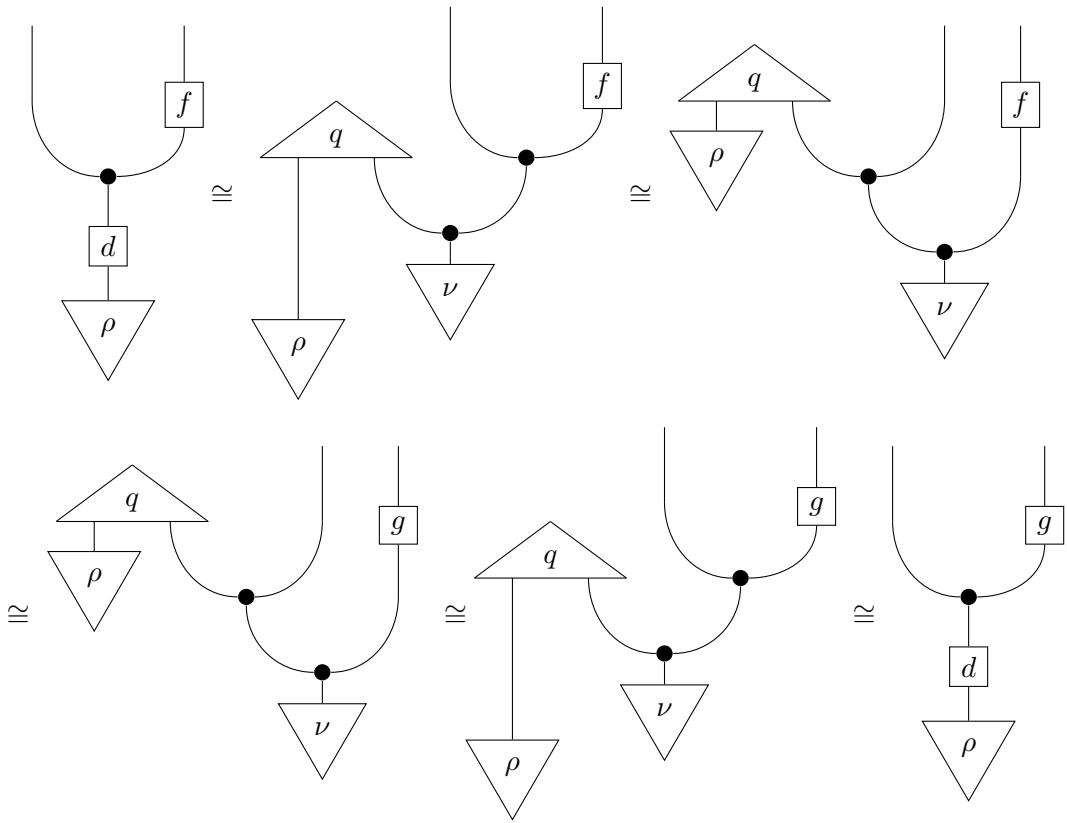
\[ \input{img/eff-blocks-almost-eq-1.tikz} \]
\[ \input{img/eff-blocks-almost-eq-2.tikz} \]
The second isomorphism holds by the associativity of $\bcopier$ \eqref{eq:comonoid-law}; the third since $f \overset{\nu}{\sim} g$; the fourth by associativity of $\bcopier$; and the fifth by the density function representation for $d$. This establishes the required relation.
\end{proof}

\section{Review of basic coend calculus}
\label{sec:orgb43af07}
\label{sec:coends}

In \secref{sec:optics}, we introduced coends informally in the context of optics. Here, we briefly review some basic properties of coends and their associated \emph{coend calculus}. We continue to work in the general setting of a \(\Cat{V}\text{-enriched}\) category \(\cat{C}\), where \(\Cat{V}\) is assumed to be cocomplete and Cartesian closed.

\paragraph{(Co)presheaves act by composition}
\label{sec:orgcef9436}

Recall that for every object \(X : \cat{C}_0\), we have the representable presheaf \(\cat{C}(-, X) : \cat{C}\op \to \Cat{V}\) and representable copresheaf \(\cat{C}(X, -) : \cat{C} \to \Cat{V}\). \(\cat{C}(-,X)\) represents the collection of morphisms \emph{into} X (\emph{i.e.} with codomain X), and \(\cat{C}(X,-)\) the collection of morphisms \emph{out of} X (\emph{i.e.} with domain X).

Being functors, representable (co)presheaves act on objects and morphisms in \(\cat{C}\). On objects, \(\cat{C}(X, Y)\) is the object in \(\Cat{V}\) of morphisms \(X \klto Y\) in \(\cat{C}\). Given a morphism \(h : \cat{C}(Y, Z)\), the copresheaf \(\cat{C}(X, -)\) takes \(h\) to the map \(\cat{C}(X, h) : \cat{C}(X, Y) \to \cat{C}(X, Z)\) in \(\Cat{V}\) which acts by postcomposition: that is, \(\cat{C}(X, h)\) takes \(g : \cat{C}(X, Y)\) to \(h \klcirc g : \cat{C}(X, Z)\). Similarly, the action of the presheaf \(\cat{C}(-, Y)\) is precomposition: given a morphism \(f : \cat{C}(W, X)\), the map \(\cat{C}(f, Y) : \cat{C}(X, Y) \to \cat{C}(W, Y)\) takes \(g : \cat{C}(X, Y)\) to \(g \klcirc f : \cat{C}(W, Y)\). Note that the action of a presheaf is contravariant: presheaves `pull back'; copresheaves `push forwards'.

\paragraph{Profunctors}
\label{sec:org87200b5}

By allowing the functors \(\cat{C}(X, -)\) and \(\cat{C}(-, Y)\) to vary in both arguments simultaneously, we obtain the \(\mathsf{hom}\) \emph{profunctor} \(\cat{C}(-,=) : \cat{C}\op \times \cat{C} \to \Cat{V}\), which picks out the object \(\cat{C}(X, Y)\) in \(\Cat{V}\) of morphisms \(X \klto Y\) in \(\cat{C}\) for each pair of objects \((X, Y)\) in \(\cat{C}\); we can think of the elements this hom object as witnessing the relation between \(X\) and \(Y\), and indeed, profunctors are categorified relations
(see \textcite[Example 5.4]{Loregian2015This} or \textcite[Chapter 4]{Fong2018Seven}).
Naturally, the action of the hom profunctor on morphisms in \(\cat{C}\) is by pulling back (precomposition) on the left and pushing forwards (postcomposition) on the right.

\paragraph{Coends}
\label{sec:org3978926}

Fix an arbitrary profunctor \(P : \cat{C}\op \times \cat{C} \to \Cat{V}\). The coend \(\int^{\cat{C}} P\) is the greatest quotient in \(\Cat{V}\) of \(\coprod_X P(X, X)\) that coequalizes the left and right actions of \(P\) on morphisms in \(\cat{C}\). \(\coprod_X P(X,X)\) is the coproduct (disjoint union) of all the objects \(P(X,X)\). That is, the elements of the coend \(\int^{\cat{C}} P\) are equivalence classes such that two elements \(u : P(W, W)\) and \(v : P(X, X)\) are related if there exists some \(f : P(X, W)\) satisfying \(u \klcirc f = f \klcirc v\); that is, \(u\) and \(v\) are related by the coend if we can `slide' some \(f\) between them. Compare this with the discussion following the definition of optics, Definition \ref{def:optics}. The universal property of the coend is then that it is the largest such quotient, so that any morphism out of any such quotient factors through the coend.

More formally, let \(f : X \klto Y\) be any morphism in \(\cat{C}\); we will denote its domain \(X\) by \(\dom(f) = X\) and its codomain \(Y\) by \(\cod(f) = Y\). We can then formally define the coend \(\int^{\cat{C}} P\) as the following coequalizer.

\begin{defn}[Coend] \label{def:coend}
Define witnesses to the left and right actions of \(P\) as follows.
\begin{equation*}
\lambda^\ast := \coprod_{\substack{W : \cat{C}, \, X : \cat{C}, \\ f \, : \, \cat{C}(W, \, X)}} \iota_{\dom(f)} \circ P\big(f, \dom(f)\big)
\hspace{0.04\linewidth}
\text{and}
\hspace{0.04\linewidth}
\rho^\ast := \coprod_{\substack{W : \cat{C}, \, X : \cat{C}, \\ f \, : \, \cat{C}(W, \, X)}} \iota_{\cod(f)} \circ P\big(\cod(f), f\big)
\end{equation*}
where, given an object \(Y : \cat{C}_0\), we denote by \(\iota_Y : P(Y, Y) \to \coprod_X P(X, X)\) the inclusion of \(P(Y,Y)\) into the coproduct \(\coprod_X P(X,X)\). Then the coend \(\int^{\cat{C}} P\) is given by the following coequalizer:
\begin{equation*}
\begin{tikzcd}
{\displaystyle\coprod_{\substack{W : \cat{C}, \, X : \cat{C}, \\ f \, : \, \cat{C}(W, \, X)}} P\big(\cod(f), \dom(f)\big)} \arrow[rr, "\lambda^{\ast}", shift left=2] \arrow[rr, "\rho^{\ast}"', shift right=2] &  & {\displaystyle\coprod_X P(X, X)} \arrow[rr, two heads] &  & \displaystyle\int^{\cat{C}} P
\end{tikzcd}
\end{equation*}
To make the `variable of integration' explicit, we often denote the coend \(\int^{\cat{C}} P\) by
\[ \int^{X : \, \cat{C}} P(X, \, X) \]
\end{defn}

\paragraph{Basic coend calculus}
\label{sec:org0590da6}

We make use of the following isomorphisms, which are equivalent to the basic categorical result that any (co)presheaf is canonically a colimit of representables
\parencite[Theorem 6.2.17]{Leinster2016Basic}.
They are easy to prove using the Yoneda lemma and the basic categorical result that \(\mathsf{hom}\) functors preserve (co)limits; for an explicit proof, we refer the reader to \citet{Loregian2015This}. The application of these to the proof of categorical results, along with the application of universal properties and the Yoneda lemma, is known as \emph{coend calculus}.

\begin{prop}[Yoneda reduction] \label{prop:coyoneda}
Let \(F : \cat{C}\op \to \Cat{V}\) be any presheaf, and \(G : \cat{C} \to \Cat{V}\) any presheaf. Then
\begin{equation} \label{eq:coyoneda}
F \cong \int^{X : \, \cat{C}} \cat{C}(-, X) \times F(X)
\hspace{0.07\linewidth}
\text{and}
\hspace{0.07\linewidth}
G \cong \int^{X : \, \cat{C}} G(X) \times \cat{C}(X, -) \; .
\end{equation}
\end{prop}

\begin{cor*}
The \(\mathsf{hom}\) profunctor is the coend analogue of the Dirac delta measure: for any \(W,Y : \cat{C}\),
\begin{equation*}
\cat{C}(-, Y) \cong \int^{X : \, \cat{C}} \cat{C}(-, X) \times \cat{C}(X, Y)
\hspace{0.04\linewidth}
\text{and}
\hspace{0.04\linewidth}
\cat{C}(W, -) \cong \int^{X : \, \cat{C}} \cat{C}(W, X) \times \cat{C}(X, -) \; .
\end{equation*}
\end{cor*}

Along with the idea of presheaves as `pulling back' by precomposition and copresheaves as `pushing forwards' by postcomposition, these last two isomorphisms supply useful intuition about Yoneda reduction. The first says roughly that the object of maps into \(Y\) is the same as the object of maps into some \(X\) paired with maps into \(Y\) such that we can slide forwards by composition from \(X\) into \(Y\); for instance, we can choose the representative given by \(X = Y\) and where the second map in the pair is \(\id_Y\). The second isomorphism is just the dual, where instead of sliding forwards, we slide back.
\end{document}